\documentclass{article}
\usepackage{amssymb}
\usepackage{amsfonts}
\usepackage{amsmath}
\usepackage{color}

\newtheorem{theorem}{Theorem}

\newtheorem{corollary}[theorem]{Corollary}

\newtheorem{lemma}[theorem]{Lemma}

\newenvironment{proof}[1][Proof]{\noindent\textbf{#1.} }{\ \rule{0.5em}{0.5em}}

\begin{document}

\title{On the survival of a class of subcritical branching processes in
random environment}
\author{Vincent Bansaye\thanks{%
CMAP, Ecole Polytechnique, Route de Saclay, 91128 Palaiseau Cedex, France;
e-mail: bansaye@polytechnique.edu}\, and Vladimir Vatutin\thanks{%
Department of Discrete Mathematics, Steklov Mathematical Institute, 8,
Gubkin str., 119991, Moscow, Russia; e-mail: vatutin@mi.ras.ru}, }
\maketitle

\begin{abstract}
Let $Z_{n}$ be the number of individuals in a subcritical BPRE evolving in
the environment generated by iid probability distributions. Let
$X$ be the logarithm of the expected offspring size per individual given the
environment. Assuming that the density of $X$ has the form
\begin{equation*}
p_{X}(x)=x^{-\beta -1}l_{0}(x)e^{-\rho x}
\end{equation*}%
for some $\beta >2,$ a slowly varying function $l_{0}(x)$ and $\rho \in
\left( 0,1\right) ,$ we find the asymptotic of the survival probability $%
\mathbb{P}\left( Z_{n}>0\right) $ as $n\rightarrow \infty $, prove a Yaglom
type conditional limit theorem for the process and describe the conditioned
environment. The survival probability decreases exponentially with an
additional polynomial term related to the tail of $X$. The proof uses in
particular a fine study of a random walk (with negative drift and heavy
tails) conditioned to stay positive until time $n$ and to have a small
positive value at time $n$, with $n\rightarrow \infty $.
\end{abstract}

\section{Introduction}

We consider the model of branching processes in random environment
introduced by Smith and Wilkinson \cite{SmWil}. The formal definition of
these \ processes looks as follows. Let $\mathfrak{N}$ be the space of
probability measures on $\mathbb{N}_{0}=\{0,1,2,...\}$. Equipped with the
metric of total variation $\mathfrak{N}$ becomes a Polish space. Let $%
\mathfrak{e}$ be a random variable taking values in $\mathfrak{N}$. An
infinite sequence $\mathcal{E}=(\mathfrak{e}_{1},\mathfrak{e}_{2},\ldots )$
of i.i.d. copies of $\mathfrak{e}$ is said to form a \emph{random environment%
}. A sequence of $\mathbb{N}_{0}$-valued random variables $%
Z_{0},Z_{1},\ldots $ is called a \emph{branching process in the random
environment} $\mathcal{E}$, if $Z_{0}$ is independent of $\mathcal{E}$ and,
given $\mathcal{E},$ the process $Z=(Z_{0},Z_{1},\ldots )$ is a Markov chain
with
\begin{equation}
\mathcal{L}\left( Z_{n}\;|\;Z_{n-1}=z_{n-1},\,\mathcal{E}=(e_{1},e_{2},%
\ldots )\right) \ =\ \mathcal{L}\left( \xi _{n1}+\cdots +\xi
_{nz_{n-1}}\right)  \label{transition}
\end{equation}%
for every $n\geq 1,\,z_{n-1}\in \mathbb{N}_{0}$ and $e_{1},e_{2},\ldots \in
\mathfrak{N}$, where $\xi _{n1},\xi _{n2},\ldots $ are i.i.d. random
variables with distribution $\mathfrak{e}_{n}$. Thus,
\begin{equation}
Z_{n}=\sum_{i=1}^{Z_{n-1}}\xi _{ni}  \label{defZ}
\end{equation}%
and, given the environment, $Z$ is an ordinary inhomogeneous Galton-Watson
process. We will denote the corresponding probability measure and
expectation on the underlying probability space by $\mathbb{P}$ and $\mathbb{%
E}$, respectively.

Let%
\begin{equation*}
X=\log \left( \sum_{k\geq 0}k\mathfrak{e}\left( \{k\}\right) \right) ,\qquad
X_{n}=\log \left( \sum_{k\geq 0}k\mathfrak{e}_{n}\left( \{k\}\right) \right)
,\,n=1,2,...,\
\end{equation*}%
be the logarithms of the expected offspring size per individual in the
environments and
\begin{equation*}
S_{0}=0,\ S_{n}=X_{1}+\cdots +X_{n},n\geq 1,
\end{equation*}%
be their partial sums. \newline

This paper deals with the subcritical branching processes in random
environment, i.e., in the sequel we always assume that
\begin{equation}
\mathbb{E}\left[ X\right] =-b<0.  \label{Expect}
\end{equation}%
The subcritical branching processes in random environment admit an
additional classification, which is based on the properties of the moment
generating function
\begin{equation*}
\varphi (t)=\mathbb{E}\left[ e^{tX}\right] =\mathbb{E}\left[ \left(
\sum_{k\geq 0}k\mathfrak{e}\left( \{k\}\right) \right) ^{t}\right] ,\quad
t\geq 0.
\end{equation*}%
Clearly, $\varphi ^{\prime }(0)=$ $\mathbb{E}\left[ X\right] $. Let
\begin{equation*}
\rho _{+}=\sup \left\{ t\geq 0:\varphi (t)<\infty \right\}
\end{equation*}%
and $\rho _{min}$ be the point where $\varphi (t)$ attains its minimal value
on the interval $[0,\rho _{+}\wedge 1]$. Then a subcritical branching
process in random environment is called
\begin{equation*}
\begin{array}{c}
\text{weakly subcritical if }\rho _{min}\in \left( 0,\rho _{+}\wedge
1\right) ,\qquad \qquad \qquad \qquad \qquad \qquad \qquad \ \ \  \\
\text{intermediately subcritical if }\rho _{min}=\rho _{+}\wedge 1>0\text{ \
\ and \ \ \ \ \ \ \ \ }\varphi ^{\prime }(\rho _{min})=0,\  \\
\text{strongly subcritical \ if \ \ \ \ \ \ \ }\rho _{min}=\rho _{+}\wedge 1%
\text{ \ \ \ \ \ \ \ and \ \ \ \ \ \ \ \ }\varphi ^{\prime }(\rho _{min})<0.%
\end{array}%
\end{equation*}

Note that this classification is slightly different from that given in \cite%
{bgk}. Weakly subcritical and intermediately subcritical branching processes
have been studied in \cite{gkv, af98, abkv1, abkv2} in detail. Let us recall
that $\varphi ^{\prime }(\rho _{+}\wedge 1)>0$ for the weakly subcritical
case. \newline
The strongly subcritical case is also well studied for the case $\rho
_{+}\geq 1$, i.e., if $\rho _{min}=\rho _{+}\wedge 1=1$ and $\varphi
^{\prime }(1)<0$. In particular, it was shown in \cite{gkv} and refined in
\cite{agkv2} that if $\varphi ^{\prime }(1)=\mathbb{E}\left[ Xe^{X}\right]
<0 $ and $\mathbb{E}\left[ Z_{1}\log ^{+}Z_{1}\right] <\infty $ then, as $%
n\rightarrow \infty $%
\begin{equation}
\mathbb{P}\left( Z_{n}>0\right) \sim K\left( \mathbb{E}\left[ \xi \right]
\right) ^{n},\,K>0,  \label{AsQQ}
\end{equation}%
and, in addition,%
\begin{equation}
\lim_{n\rightarrow \infty }\mathbb{E}\left[ s^{Z_{n}}|Z_{n}>0\right] =\Psi
(s),  \label{Limsub}
\end{equation}%
where $\Psi (s)$ is the probability generating function of a proper
nondegenerate random variable on $\mathbb{Z}_{+}$. This statement is
actually an extension of the classical result for the ordinary subcritical
Galton-Watson branching processes.

\section{Main results}

Our main concern in this paper is the strongly subcritical branching
processes in random environment with $\rho _{+}\in (0,1)$. More precisely,
we assume that the following condition is valid: \newline

\textbf{Hypothesis A.} The distribution of $X$ has density \textbf{\ }%
\begin{equation}
p_{X}\left( x\right) =\frac{l_{0}(x)}{x^{\beta +1}}e^{-\rho x},
\label{remainder0}
\end{equation}%
where $l_{0}(x)$ is a function slowly varying at infinity, $\beta >2,$ $\rho
\in (0,1)$ and, in addition, \newline
\begin{equation}
\quad \varphi ^{\prime }(\rho )=\mathbb{E}\left[ Xe^{\rho X}\right] <0.
\label{Ctail1}
\end{equation}

This assumption can be relaxed by assuming that $p_{X}(x)$ is the density of
$X$ for $x$ large enough, or that the tail distribution $\mathbb{P}(X\in
\lbrack x,x+\Delta))\sim \int_{x}^{x+\Delta}p_{X}(y)dy$ for $x\rightarrow
\infty $ (uniformly with respect to $\Delta\leq 1$). \newline
\ \ \ \ Clearly, $\rho =\rho _{+}<1$ under Hypothesis A. Observe that the
case $\rho =\rho _{+}=0$ not included in Hypothesis A has been studied in
\cite{vz} and yields a new type of the asymptotic behavior of subcritical
branching processes in random environment. Namely, it was established that,
as $n\rightarrow \infty $%
\begin{equation}
\mathbb{P}\left( Z_{n}>0\right) \sim K\mathbb{P}\left( X>nb\right) =K\frac{%
l_{0}(nb)}{\left( nb\right) ^{\beta }},\,K>0,  \label{casVZ}
\end{equation}%
so that the survival probability decays with a polynomial rate only.
Moreover, for any $\varepsilon >0,$ some constant $\sigma >0$ and any $x\in
\mathbb{R}$%
\begin{equation*}
\mathbb{P}\left( \frac{\log Z_{n}-\log Z_{\left[ n\varepsilon \right]
}+n\left( 1-\varepsilon \right) b}{\sigma \sqrt{n}}\leq x\,|\,Z_{n}>0\right)
=\mathbb{P}\left( B_{1}-B_{\varepsilon }\leq x\right)
\end{equation*}%
where $B_{t}$ is a standard Brownian motion. Therefore, given the survival
of the population up to time $n$, the number of individuals in the process
at this moment tends to infinity as $n\rightarrow \infty $ that is not the
case for other types of subcritical processes in random environment. \newline

The goal of the paper is to investigate the asymptotic behavior of the
survival probability of the process meeting Hypothesis A and to prove a
Yaglom-type conditional limit theorem for the distribution of the number of
individuals. To this aim we additionally assume that the sequence of
conditional probability measures
\begin{equation*}
\mathbb{P}^{\left[ x\right] }\left( \cdot \right) =\mathbb{P}\left( \cdot \
|\ X=x\right)
\end{equation*}%
is well defined for $x\rightarrow \infty $ under Hypothesis $A$. We provide
in Section \ref{Ex} natural examples when this assumption and Hypothesis $B$
below are valid.

Denote by $\mathfrak{L}=\{\mathcal{L}\}$ the set of all proper probability
measures $\mathcal{L}(\cdot )$ of nonnegative random variables. Our next
condition concerns the behavior of the measures $\mathbb{P}^{\left[ x\right]
}$ as $x\rightarrow \infty:$ \newline

\textbf{Hypothesis B.} There exists a probability measure $\mathbb{P}^{\ast
} $ on\textbf{\ } $\mathfrak{L}$ such that, as $x\rightarrow \infty$,
\begin{equation*}
\mathbb{P}^{\left[ x\right] }\Longrightarrow \mathbb{P}^{\ast }
\end{equation*}%
where the symbol $\Longrightarrow $ stands for the weak convergence of
measures.

Setting
\begin{equation*}
a=-\frac{\varphi ^{\prime }\left( \rho \right) }{\varphi \left( \rho \right)
}>0,
\end{equation*}%
we are now ready to formulate the first main result of the paper.

\begin{theorem}
\label{T_Extin} If
\begin{equation}
\mathbb{E}\left[ -\log \left( 1-\mathfrak{e}\left( \{0\}\right) \right) %
\right] <\infty ,\quad \mathbb{E}\left[ e^{-X}\sum_{k\geq 1}\mathfrak{e}%
\left( \{k\}\right) k\log k\right] <\infty  \label{Mom}
\end{equation}%
and Hypotheses A and B are valid, then there exists a constant $C_{0}>0$
such that, as $n\rightarrow \infty $ 
\begin{equation}
\mathbb{P}\left( Z_{n}>0\right) \sim C_{0}\rho \varphi ^{n-1}\left( \rho
\right) e^{an\rho }\mathbb{P}\left( X>an\right) \sim C_{0}\rho \varphi
^{n-1}\left( \rho \right) \frac{l_{0}(n)}{\left( an\right) ^{\beta +1}}.
\label{AsTotal}
\end{equation}
\end{theorem}

We stress that $\varphi (\rho )\in (0,1)$. Moreover, the explicit form of $%
C_{0}$ can be found in (\ref{C_finsurviv}). The proof is given in Section %
\ref{ProofMT}. We now quickly explain this asymptotic behavior and give at
the same time an idea of the proof. In the next Section, some examples of
processes satisfying the assumptions required in Theorem~\ref{T_Extin} can
be found.

For the proof, we introduce in Section \ref{NewP} a new probability measure $%
\mathbf{P}$. Under this new probability measure, the random walk $\mathbf{S}%
=\left( S_{n},n\geq 0\right) $ has the drift $-a<0$ and the heavy tail
distribution of its increments has polynomial decay $\beta .$ Adding that $%
\mathbf{E}\left[ \exp (\rho X)\right] =\varphi \left( \rho \right) $, we
will get the survival probability as
\begin{equation*}
\varphi ^{n}\left( \rho \right) \mathbf{E}\left[ e^{-\rho S_{n}}\mathbf{P}%
(Z_{n}>0|\mathfrak{e})\right] \approx const\times \varphi ^{n}\left( \rho
\right) \mathbf{P}(L_{n}\geq 0,S_{n}\leq N)
\end{equation*}%
where $L_{n}$ is the minimum of the random walk up to time $n$ and $N$ is
(large but) fixed.

We then make use of the properties of random walks with negative drift and
heavy tails of increments established in \cite{VVI} to show that
\begin{equation*}
\mathbf{P}(L_{n}\geq 0,S_{n}\leq N)\approx const\times \mathbf{P}(X_{1}\in
\lbrack an-M\sqrt{n},an+M\sqrt{n}],S_{n}\in \lbrack 0,1])
\end{equation*}%
for $n$ large enough and conclude using the central limit theorem.

$\newline
$

Our second main result is a Yaglom-type conditional limit theorem.

\begin{theorem}
\label{T_condEnd} Under the conditions of Theorem \ref{T_Extin},
\begin{equation*}
\lim_{n\rightarrow \infty }\mathbb{E}\left[ s^{Z_{n}}|Z_{n}>0\right] =\Omega
(s),
\end{equation*}%
where \ $\Omega (s)$ is the probability generating function of a proper
nondegenerate random variable on $\mathbb{Z}_{+}$.
\end{theorem}

We see that, contrary to the case $\rho _{min}=\rho _{+}\wedge 1=0$ \ \ this
Yaglom-type limit theorem has the same form as for the ordinary
Galton-Watson subcritical processes.

Introduce a sequence of generating functions
\begin{equation*}
f_{n}(s)=f(s;\mathfrak{e}_{n})\ =\ \sum_{k=0}^{\infty }\,\mathfrak{e}%
_{n}(\{k\})s^{k},\qquad 0\leq s\leq 1,
\end{equation*}%
specified by the environmental sequence $\left( \mathfrak{e}_{1},\mathfrak{e}%
_{2},...,\mathfrak{e}_{n},...\right) $ and denote
\begin{equation}
f_{j,n}=f_{j+1}\circ \cdots \circ f_{n},\quad f_{n,j}=f_{n}\circ \cdots
\circ f_{j+1}\quad (j<n),~f_{n,n}=Id.  \label{NotIterat}
\end{equation}
For every pair $n\geq j\geq 1$, we define a random function $g_{j}:\mathbb{R}%
_{+}\rightarrow \left[ 0,1\right] ,$ a tuple of random variables
\begin{equation}
W_{n,j}=\frac{1-f_{n,j}(0)}{e^{S_{n}-S_{j}}}  \label{DefWn}
\end{equation}%
and a random variable $W_{j}$ on $[0,1]$ such that

\begin{itemize}
\item[(i)] the distribution of $g_{j}$ is given by $\mathbb{P}^{\ast }$ and
that of $W_{j}$ is given by the (common) distribution of $\lim_{n\rightarrow
\infty }W_{n,j}$, which exists by monotonicity;

\item[(ii)] $f_{0,j-1}$, $g_{j}$ and $(W_{n,j},W_{j},f_{k}:k\geq j+1)$ are
independent for each $n\geq j$
(it is always possible, the initial probability
space being extended if required).
\end{itemize}

Then we can set
\begin{equation*}
c_{j}=\int_{-\infty }^{\infty }\mathbb{E}\left[
1-f_{0,j-1}(g_{j}(e^{v}W_{j}))\right] e^{-\rho v}dv
\end{equation*}%
and state the following result. It describes the environments that provide
survival of the population until time $n$. 

\begin{theorem}
For each $j\geq 1$,\newline
i)  the following  limit exists
$$\pi _{j}=\lim_{n\rightarrow \infty }\mathbb{P}(X_{j}\geq
an/2|Z_{n}>0)=\frac{c_{j}\varphi ^{-j}(\rho )}{\sum_{k\geq 1}c_{k}\varphi ^{-k}(\rho )}.$$
ii) 
for each measurable and bounded function $F:\mathbb{R}%
^{j}\rightarrow \mathbb{R}$ and each family of measurable uniformly bounded
functions $F_{n}:\mathbb{R}^{n+1}\rightarrow \mathbb{R}$ the difference
\begin{eqnarray*}
&&\mathbb{E}\left[ F(S_{0},\ldots
,S_{j-1})F_{n-j}(S_{n}-S_{j-1},X_{j+1},\ldots ,X_{n})|Z_{n}>0,\ X_{j}\geq
an/2\right] \\
&&\qquad -\,c_j^{-1}\mathbb{E}\left[ F(S_{0},\ldots
,S_{j-1})\int_{-\infty }^{\infty }F_{n-j}(v,X_n,...,X_{j+1})G_{j,n}(v)dv%
\right]
\end{eqnarray*}%
goes to $0$ as $n\rightarrow \infty $, where
\begin{equation*}
G_{j,n}(v):=\left( 1-f_{0,j-1}(g_{j}(e^{v}W_{n,j}))\right) e^{-\rho v}.
\end{equation*}
\end{theorem}

Detailed descriptions of the properties of the random function $g_j$ and the
random variable $W$ are given by (\ref{Deffg}) and before the proof of Lemma %
\ref{limO}, respectively. We refer to \cite{abkv1, abkv2,agkv, agkv2} for
similar questions in the subcritical and critical regimes. Here the
conditioned environment is different since a big jump appear, whereas the
rest of the random walk looks like the original one. Let us now focus on
this exceptional environment explaining the survival event and give a more
explicit result.

\begin{corollary}
Let $\varkappa =\inf \{j\geq 1:X_{j}\geq an/2\}$. Under $\mathbb{P}$,
conditionally on $Z_{n}>0$, $\varkappa $ converges in distribution to a
proper random variable whose distribution is given by $(\pi _{j}:j\geq 1)$.
Moreover, conditionally on $\{Z_n >0, X_j \geq an/2\}$, the distribution law
of $(X_{\varkappa }-an)/(VarX\sqrt{n})$ converges to a law $\mu $ specified
by
\begin{equation*}
\mu (B)=\,c_j^{-1}\mathbb{E}\left[ 1(G\in B)\int_{-\infty
}^{\infty }\left( 1-f_{0,j-1}(g_{j}(e^{v}W_{j}))\right) e^{-\rho v}dv\right]
\end{equation*}%
for any Borel set $B\subset \mathbb{R}$, where $G$ is a centered gaussian
random variable with variance $Var X$, which is independent of $%
(f_{0,j-1},g_j)$.
\end{corollary}

\section{Examples}

\label{Ex}

We provide here some examples meeting the conditions of Theorem \ref{T_Extin}%
. Thus, we assume that Hypothesis $A$ is valid and we focus on the existence
and convergence of $\mathbb{P}^{\left[ x\right] }$. Let us first deal with
the existence of random reproduction laws $\mathfrak{e}$ for which the
conditional probability
\begin{equation*}
\mathbb{P}^{\left[ x\right] }\left( \cdot \right) =\mathbb{P}\left( \cdot \
|\ X=x\right)
\end{equation*}%
is well defined. \newline
\newline
\emph{Example 0.} Assume that the environment $\mathfrak{e}$ takes its
values in some set $\mathcal{M}$ of probability measures such that for all $%
\mu ,\nu \in \mathcal{M}$
\begin{equation*}
\sum_{k\geq 0}k\mu (k)<\sum_{k\geq 0}k\nu (k)\Rightarrow \mu \leq \nu ,
\end{equation*}%
where $\mu \leq \nu $ means that $\forall l\in \mathbb{N},\mu \lbrack
l,\infty )\leq \nu \lbrack l,\infty )$. We note that Hypothesis $A$ ensures
that $\mathbb{P}(\cdot |X\in \lbrack x,x+\epsilon ))$ is well defined. Then,
for every $H$ $:\mathcal{M}\rightarrow \mathbb{R}^{+}$ which is non
decreasing in the sense that $\mu \leq \nu $ implies $H(\mu )\leq H(\nu )$,
we get that the functional
\begin{equation*}
\mathbb{E}\left[ H(\mathfrak{e})|X\in \lbrack x,x+\epsilon )\right]
\end{equation*}%
decreases to some limit $p(H)$ as $\epsilon \rightarrow 0.$ Thus, writing $%
H_{l,y}(\mu )=1$ if $\mu \lbrack l,\infty )\geq y$ and $0$ otherwise, we can
define $\mathbb{P}^{\left[ x\right] }$ via
\begin{equation*}
\mathbb{P}^{\left[ x\right] }\left( \mathfrak{e}[l,\infty )\geq y\right)
=p(H_{l,y})
\end{equation*}%
to get the expected conditional probability. \newline

Let us now focus on Hypothesis B. \newline
\newline
\emph{Example 1. }Let $f(s;\mathfrak{e})=\sum_{k\geq 0}\mathfrak{e}\left(
\{k\}\right) s^{k}$ be the (random) probability generating function
corresponding to the random measure $\mathfrak{e\in N}$ and let (with a
slight abuse of notation) $\xi =\xi \left( \mathfrak{e}\right) \geq 0$ be
the integer-valued random variable with probability generating function $f(s;%
\mathfrak{e})$, i.e., $f(s;\mathfrak{e})=E\left[ s^{\xi \left( \mathfrak{e}%
\right) }\right] $.

It is not difficult to understand that if $\mathbb{E}\left[ \log f^{\prime
}(1;\mathfrak{e})\right] <0$ and there exists a deterministic function $%
g(\lambda ),\lambda \geq 0,$ with $g(\lambda )<1,\lambda >0,$ and $g(0)=1,$
such that, for every $\varepsilon >0$
\begin{equation*}
\lim_{y\rightarrow \infty }\mathbb{P}\left( \mathfrak{e}:\sup_{0\leq \lambda
<\infty }\left\vert f\left( e^{-\lambda /y};\mathfrak{e}\right) -g(\lambda
)\right\vert >\varepsilon \,\Big|\,\,f^{\prime }(1;\mathfrak{e}%
)=y=e^{x}\right) =0,
\end{equation*}%
then Hypothesis B is satisfied for the respective subcritical branching
process. \newline

We now give two more explicit examples for which Hypothesis B holds true and
note that mixing the two classes described in these examples would provide a
more general family which satisfies Hypothesis B. \newline

Let $\mathfrak{N}_{f}\,\mathfrak{\subset N}$ be the set of probability
measures on $\mathbb{N}_{0}$ such that
\begin{equation*}
e=e\left( t,y\right) \in \mathfrak{N}_{f}\iff f\left( s;e\right) =1-t+\frac{t%
}{1+yt^{-1}\left( 1-s\right) }
\end{equation*}%
where $t\in (0,1]$ and $y\in \left( 0,\infty \right) ,$ and let $\mathfrak{L}%
_{g}\subset \mathfrak{L}$ be the set of probability measures
such that%
\begin{equation*}
L=L\left( t,y\right) \in \mathfrak{L} _{g}\iff g(t,\lambda )=\int
e^{-\lambda y}L(t,dy)=1-t+\frac{t^{2}}{t+\lambda }.
\end{equation*}%
Let, further, $\mathcal{B=B}_{1}\times \mathcal{B}_{2}\mathcal{\subset }%
(0,1] $ $\times \left( 0,\infty \right) $ be a Borel set. We write
\begin{equation*}
e=e\left( t,y\right) \in T\left( \mathcal{B}\right) \,\mathcal{\subseteq \,}%
\mathfrak{N}_{f}\text{ if }\left( t,y\right) \in \mathcal{B}
\end{equation*}%
and%
\begin{equation*}
L=L(t,y)\in T\left( \mathcal{B}_{1}\right) \subseteq \mathfrak{L}_{g}\text{
if }t\in \mathcal{B}_{1}.
\end{equation*}%
Let $\left( \theta ,\zeta \right) $ be a pair of random variables with
values in $(0,1]\times \left( 0,\infty \right) $ such that for a measure ${P}%
^{\ast }\left( \cdot \right) $ with support on $(0,1]$ and any Borel set $%
\mathcal{B}_{1}\subseteq (0,1]$,
\begin{equation*}
\lim_{x\rightarrow \infty }{P}\left( \theta \in \mathcal{B}_{1}|\zeta
=x\right) ={P}^{\ast }\left( \theta \in \mathcal{B}_{1}\right)
\end{equation*}%
exists. \newline

With this notation in view we describe the desired two examples.\newline
\emph{Example 2. } Assume that the support of the probability measure $%
\mathbb{P}\,$\ is concentrated on the set $\mathfrak{N}_{f}$ only and the
random environment $\mathfrak{e}$ is specified by the relation
\begin{equation*}
\mathfrak{e=}e\left( \theta ,\zeta \right) \iff f\left( s;\mathfrak{e}%
\right) =1-\theta +\frac{\theta ^{2}}{\theta +\zeta \left( 1-s\right) }.
\end{equation*}%
Clearly, $\log f^{\prime }\left( 1;\mathfrak{e}\right) =\log \zeta $. Thus,
\begin{equation*}
\mathbb{P}\left( e\left( \theta ,\zeta \right) \in T\left( \mathcal{B}%
\right) \right) =\mathbb{P}\left( f\left( s;\mathfrak{e}\right) :\left(
\theta ,\zeta \right) \in \mathcal{B}\right)
\end{equation*}%
and if $\mathcal{B=B}_{1}\times \left\{ x\right\} $ then
\begin{eqnarray*}
&&\lim_{x\rightarrow \infty }\mathbb{P}\left( f\left( e^{-\lambda \zeta
^{-1}};\mathfrak{e}\right) :\left( \theta ,\zeta \right) \in \mathcal{B}%
|\zeta =e^{x}\right) ={P}^{\ast }\left( \theta \in \mathcal{B}_{1}\right) \\
&&\qquad \qquad =\mathbb{P}^{\ast }\left( g(\lambda ;\theta ):\theta \in
\mathcal{B}_{1}\right) =\mathbb{P}^{\ast }\left( L(\theta ;y)\in T\left(
\mathcal{B}_{1}\right) \right) .
\end{eqnarray*}%
Note that if ${P}\left( \theta =1|\zeta =x\right) =1$ for all sufficiently
large $x$ we get a particular case of Example 1. \newline
\ \ \ \newline
\emph{Example 3.} If the support of the environment is concentrated on
probability measures $\mathfrak{e\in N}$ such that, for any $\varepsilon >0$%
\begin{equation}
\lim_{y\rightarrow \infty }\mathbb{P}\left( \mathfrak{e}:\left\vert \frac{%
\xi (\mathfrak{e})}{f^{\prime }(1;\mathfrak{e})}-1\right\vert >\varepsilon
\,\,\Big|\,f^{\prime }(1;\mathfrak{e})=e^{X}=y\right) =0  \label{Den}
\end{equation}%
and the density of the random variable $X=\log f^{\prime }(1;\mathfrak{e})$
is positive for all sufficiently large $x$, then $g(\lambda )=e^{-\lambda }$%
. Condition (\ref{Den}) is satisfied if, for instance,
\begin{equation*}
\lim_{y\rightarrow \infty }\mathbb{P}\left( \mathfrak{e}:\frac{Var\xi (%
\mathfrak{e})}{\left( f^{\prime }(1;\mathfrak{e})\right) ^{2}}>\varepsilon
\,\,\Big|\,f^{\prime }(1;\mathfrak{e})=y\right) =0.
\end{equation*}

\section{Preliminaries}

\subsection{Change of probability measure}

\label{NewP} \indent A nowadays classical technique of studying subcritical
branching processes in random environment (see, for instance, \cite%
{gkv,agkv, abkv1,abkv2} ) is similar to that one used to investigate
standard random walks satisfying the Cramer condition. Namely, denote by $%
\mathcal{F}_{n}$ \ the $\sigma -$algebra generated by the tuple $\left(
\mathfrak{e}_{1},\mathfrak{e}_{2},...,\mathfrak{e}_{n};Z_{0},Z_{1},...,Z_{n}%
\right) $ and let $\mathbb{P}^{(n)}$ be the restriction of $\mathbb{P}$ to $%
\mathcal{F}_{n}$. Setting $\ $%
\begin{equation*}
m=\varphi \left( \rho \right) =\mathbb{E}\left[ e^{\rho X}\right] ,
\end{equation*}%
we introduce another probability measure $\mathbf{P}$ by the following
change of measure
\begin{equation}
d\mathbf{P}^{(n)}=m^{-n}e^{\rho S_{n}}d\mathbb{P}^{(n)},\ n=1,2,...
\label{Cmes1}
\end{equation}%
or, what is the same, for any random variable $Y_{n}$ measurable with
respect to $\mathcal{F}_{n}$ we let%
\begin{equation}
\mathbf{E}\left[ Y_{n}\right] =m^{-n}\mathbb{E}\left[ Y_{n}e^{\rho S_{n}}%
\right] .  \label{Cemes2}
\end{equation}%
By (\ref{Ctail1}),
\begin{equation}
\mathbf{E}\left[ X\right] =m^{-1}\mathbb{E}\left[ Xe^{\rho X}\right]
=\varphi ^{\prime }\left( \rho \right) /\varphi \left( \rho \right) =-a<0.
\label{SubSubcritica}
\end{equation}%
Applying a Tauberian theorem we get%
\begin{eqnarray}
A(x) &=&\mathbf{P}\left( X>x\right) =\frac{\mathbb{E}\left[ I\left\{
X>x\right\} e^{\rho X}\right] }{m}=\frac{1}{m}\int_{x}^{\infty }e^{\rho
y}p_{X}(y)dy  \notag \\
&\mathbb{=}&\frac{1}{m}\int_{x}^{\infty }\frac{l_{0}(y)dy}{y^{\beta +1}}\sim
\frac{1}{m\beta }\frac{l_{0}(x)}{x^{\beta }}=\frac{l(x)}{x^{\beta }},
\label{Tail1}
\end{eqnarray}%
where $l(x)$ is a function slowly varying at infinity. Thus, the random
variable~$X$ under the measure $\mathbf{P}$ does not satisfy the Cramer
condition and has finite variance. \newline
The density of $X$ under $\mathbf{P}$ is
\begin{equation*}
\mathbf{p}_X(x)=-A^{\prime}(x)=\frac{1}{m} \frac{l_0(x)}{x^{\beta+1}}
\end{equation*}
and it satisfies (see Theorem 1.5.2 p22 in \cite{BGT}) for each $M\geq 0$
and $\epsilon(x)\rightarrow 0$ as $x\rightarrow 0$,
\begin{equation}  \label{densUnif}
\frac{\mathbf{p}_X(x+ t\epsilon(x)x)}{\mathbf{p}_X(x)}\overset{x\rightarrow
\infty}{\longrightarrow} 1,
\end{equation}
uniformly with respect to $t\in [-M,M]$. In particular,
\begin{equation}
A(x+\Delta )-A(x)=-\frac{\Delta \beta A(x)}{x}(1+o(1))  \label{remainder}
\end{equation}%
as $x\rightarrow \infty$ and setting
\begin{equation*}
b_n=\beta \frac{A(an)}{an},
\end{equation*}
we have
\begin{equation}
b_n^{-1}\mathbf{p}_X(an+t\sqrt{n}) \overset{n\rightarrow \infty}{%
\longrightarrow} 1,  \label{densunif}
\end{equation}
uniformly with respect to $t\in [-M,M]$. \newline

Let $\mathbf{\Phi }=\{\Phi \}$ be the metric space of the Laplace transforms
$\Phi (\lambda )=\int_{0}^{\infty }e^{-\lambda u}\mathcal{L}(du),$ $%
\,\lambda \in \lbrack 0,\infty ),$ of the laws from $\mathfrak{L}$ endowed
with the metric
\begin{equation*}
d(\Phi _{1},\Phi _{2})=\sup_{2^{-1}\leq \lambda \leq 2}|\Phi _{1}(\lambda
)-\Phi _{2}(\lambda )|.
\end{equation*}%
Since \ the Laplace transform of the distribution of a nonnegative random
variable is completely determined by its values on any interval of the
positive half-line, convergence $\Phi _{n}\rightarrow \Phi $ as $%
n\rightarrow \infty $ in metric $d$ is equivalent to weak convergence $%
\mathcal{L}_{n}\overset{w}{\rightarrow }\mathcal{L}$ of the respective
probability measures.\newline
From now on, to avoid confusions we agree to use $P$ and $E$ for the symbols
of probability and expectation in the case when the respective distributions
are not associated with the measures $\mathbb{P}$ or $\mathbf{P}$. \newline
Let $\mathfrak{F}\mathcal{=}\left\{ f(s)\right\} $ be the set of all
probability generating functions of integer-valued random variables $\eta
\geq 0$, i.e. $f(s)=E\left[ s^{\eta }\right] $ and let $\mathbf{\Phi }^{(f)}$
$\subset \mathbf{\Phi }$ be the closure (in metric $d$) of the set of all
Laplace transforms of the form
\begin{equation*}
\Phi (\lambda ;f)=f\left( \exp \left\{ -\lambda /f^{\prime }(1)\right\}
\right) ,\quad f\in \mathfrak{F}.
\end{equation*}%
The probability measure $\mathbf{P}$ on $\mathfrak{N}$ generates a natural
probability measure on the metric space $\mathbf{\Phi }^{(f)}$ which we
denote by the same symbol $\mathbf{P}$. \newline
Introduce a sequence of probability measures on $\mathbf{\Phi }^{(f)}$ by
the equality
\begin{equation*}
\mathbf{P}^{\left[ x\right] }\left( \cdot \right) =\mathbf{P}\left( \cdot \
|\ f^{\prime }(1;\mathfrak{e})=e^{x}\right) .
\end{equation*}%
With this new probability measure, Hypothesis B is now equivalent to \newline

\textbf{Hypothesis B'}. There exists a \textbf{\ } measure $\mathbf{P}^{\ast
}\left( \cdot \right) $ on $\mathbf{\Phi }^{(f)}$ \textbf{(}with the support
on\textbf{\ }$\Phi (\lambda ):\Phi (0)=1,$ $\Phi (\lambda )<1$, $\lambda >0$%
) such that, as $x\rightarrow \infty $
\begin{equation*}
\mathbf{P}^{\left[ x\right] }\Longrightarrow \mathbf{P}^{\ast }.
\end{equation*}%
$\newline
$ In the other words, Hypothesis B' means that there exists a (random) a.s.
continuous on $[0,\infty )$ function $g(\cdot )$ with values in $\mathbf{%
\Phi }^{(f)}$ such that, for every continuous bounded functional $H$ on $%
\mathbf{\Phi }^{(f)}$%
\begin{equation}
\lim_{x\rightarrow \infty }\mathbf{E}^{\left[ x\right] }\left[ H(\Phi )%
\right] =\mathbf{E}^{\ast }\left[ H(g)\right] .  \label{Deffg}
\end{equation}%
Since, for any fixed $\lambda \geq 0$ the functional $H_{\lambda }(\Phi
)=\Phi (\lambda )$ is continuous on $\mathbf{\Phi }^{(f)}$, we have for $%
y=e^{x}$%
\begin{equation}
\lim_{y\rightarrow \infty }\mathbf{E}\left[ f(e^{-\lambda /y};\mathfrak{e}%
)\,|\ f^{\prime }(1;\mathfrak{e})=y\right] =\mathbf{E}^{\ast }\left[
g(\lambda )\right] ,\quad \lambda \in \lbrack 0,\infty )  \label{Unif0}
\end{equation}%
and $\mathbf{E}^{\ast }\left[ g(0)\right] =1,\mathbf{E}^{\ast }\left[
g(\lambda )\right] <1$ if $\lambda >0$. The prelimiting functions at the
left-hand side of (\ref{Unif0}) have the form%
\begin{equation*}
\mathbf{E}\left[ f(e^{-\lambda /y};\mathfrak{e})\,|\ f^{\prime }(1;\mathfrak{%
e})=y\right] =\mathbf{E}\left[ e^{-\lambda \xi (\mathfrak{e})/y}\,|\
f^{\prime }(1;\mathfrak{e})=y\right]
\end{equation*}%
and, therefore, are the Laplace transforms of the distributions of some
random variables. Hence, by the continuity theorem for Laplace transforms
there exists a proper nonnegative random variable $\theta $ such that
\begin{equation*}
\lim_{y\rightarrow \infty }\mathbf{E}\left[ f(e^{-\lambda /y};\mathfrak{e}%
)|\ f^{\prime }(1;\mathfrak{e})=y\right] =\mathbf{E}^{\ast }\left[
e^{-\lambda \theta }\right] ,\quad \lambda \in \lbrack 0,\infty ).
\end{equation*}%
Let now
\begin{equation*}
h(s)=E\left[ s^{\upsilon }\right] =\sum_{k=0}^{\infty }h_{k}s^{k},\ h(1)=1
\end{equation*}%
be the (deterministic) probability generating function of the nonnegative
integer-valued random variable $\upsilon $. Since, for any fixed $\lambda
\geq 0$ the functional $H_{\lambda ,h}(\Phi )=h\left( \Phi (\lambda )\right)
$ is continuous on $\mathbf{\Phi }^{(f)}$, we have
\begin{equation}
\lim_{y\rightarrow \infty }\mathbf{E}\left[ h\left( f(e^{-\lambda /y};%
\mathfrak{e})\right) |\ f^{\prime }(1;\mathfrak{e})=y\right] =\mathbf{E}%
^{\ast }\left[ h\left( g(\lambda )\right) \right] ,\quad \lambda \in \lbrack
0,\infty ).  \label{Unif}
\end{equation}%
The prelimiting and limiting functions are monotone and continuous on $%
[0,\infty )$. Therefore, convergence in (\ref{Unif}) is uniform in $\lambda
\in \lbrack 0,\infty )$

Further, denoting by $\xi _{i}(\mathfrak{e}),i=1,2,...$ independent copies
of $\xi (\mathfrak{e})$ we get%
\begin{eqnarray*}
\mathbf{E}\left[ h\left( f(e^{-\lambda /y};\mathfrak{e})\right) |\ f^{\prime
}(1;\mathfrak{e})=y\right] &=&\sum_{k=0}^{\infty }h_{k}\mathbf{E}\left[
f^{k}(e^{-\lambda /y};\mathfrak{e})\,|\ f^{\prime }(1;\mathfrak{e})=y\right]
\\
&=&\sum_{k=0}^{\infty }h_{k}\mathbf{E}\left[ \exp \left\{ -\frac{\lambda }{y}%
\sum_{i=1}^{k}\xi _{i}(\mathfrak{e})\right\} \,\Big|\ f^{\prime }(1;%
\mathfrak{e})=y\right] \\
&=&\mathbf{E}\left[ \exp \left\{ -\frac{\lambda }{y}\Xi \right\} \,\Big|\
f^{\prime }(1;\mathfrak{e})=y\right],
\end{eqnarray*}%
where%
\begin{equation*}
\Xi (\mathfrak{e})=\sum_{i=1}^{\upsilon }\xi _{i}(\mathfrak{e}).
\end{equation*}%
Thus, similarly to the previous arguments there exists a proper random
variable $\Theta $ such that%
\begin{equation}
\lim_{y\rightarrow \infty }\mathbf{E}\left[ \exp \left\{ -\frac{\lambda }{y}%
\Xi (\mathfrak{e})\right\} \,\Big|\,f^{\prime }(1;\mathfrak{e})=y\right] =%
\mathbf{E}^{\ast }\left[ e^{-\lambda \Theta }\right] ,\quad \lambda \in
\lbrack 0,\infty ).  \label{Unif2}
\end{equation}%
As above, this convergence is uniform with respect to $\lambda \in [0,\infty
)$. \newline

\subsection{Some useful results on random walks}

\label{RW} We pick here from \cite{VVI} several results on random walks with
negative drift and heavy tails useful for the forthcoming proofs. Recall
that $b_{n}=\beta A(an)/(an)$, and introduce three important random
variables
\begin{equation*}
M_{n}\ =\ \max (S_{1},\ldots ,S_{n})\ ,\quad L_{n}\ =\ \min (S_{1},\ldots
,S_{n}),
\end{equation*}%
and
\begin{equation*}
\tau _{n}=\min \left\{ 0\leq k\leq n:S_{k}=L_{n}\right\}
\end{equation*}
and two right-continuous functions $U:\mathbb{R}\rightarrow \mathbb{R}%
_{0}=\left\{ x\geq 0\right\} $ and $V:\mathbb{R}\rightarrow \mathbb{R}_{0}$
given by
\begin{align*}
U(x)\ & =\ 1+\sum_{k=1}^{\infty }\mathbf{P}\left( -S_{k}\leq
x,M_{k}<0\right) \ ,\quad x\geq 0, \\
V(x)\ & =\ 1+\sum_{k=1}^{\infty }\mathbf{P}\left( -S_{k}>x,L_{k}\geq
0\right) \ ,\quad x\leq 0,\
\end{align*}%
and $0$ elsewhere. In particular $U(0)=V(0)=1$. It is well-known that $%
U(x)=O(x)$ for $x\rightarrow \infty $. Moreover, $V(-x)$ is uniformly
bounded in \thinspace $x$ in view of $\mathbf{E}X<0$.

With this notation in hands we recall the following result established in
Lemma 7 of~\cite{VVI}.

\begin{lemma}
\label{LExponent}Assume that $\mathbf{E}\left[ X\right] <0$ and that $A(x)$
meets condition~$(\ref{remainder})$. Then, for any $\lambda >0$ as $%
n\rightarrow \infty $%
\begin{equation}
\mathbf{E}\left[ e^{\lambda S_{n}};\tau _{n}=n\right] =\mathbf{E}\left[
e^{\lambda S_{n}};M_{n}<0\right] \sim b_{n}\int_{0}^{\infty }e^{-\lambda
z}U(z)\,dz  \label{AsK_1}
\end{equation}%
and%
\begin{equation}
\mathbf{E}\left[ e^{-\lambda S_{n}};\tau >n\right] =\mathbf{E}\left[
e^{-\lambda S_{n}};L_{n}\geq 0\right] \sim b_{n}\int_{0}^{\infty
}e^{-\lambda z}V(-z)\,dz.  \label{AsK_2}
\end{equation}
\end{lemma}

Moreover from (19) and (20) in \cite{VVI}, we know that for $\lambda>0$
\begin{align}
b_{n}^{-1}\mathbf{E}[e^{\lambda S_{n}};M_{n}<0\,,\,S_{n}<-x]\ & \rightarrow
\ \int_{x}^{\infty }e^{-\lambda z}U(z)\,dz,  \label{Hirano4} \\
b_{n}^{-1}\mathbf{E}[e^{-\lambda S_{n}};L_{n}\geq 0\,,\,S_{n}>x]\ &
\rightarrow \ \int_{x}^{\infty }e^{-\lambda z}V(-z)\,dz.  \label{Hirano2}
\end{align}%
and gathering Lemmas 9,10,11 in \cite{VVI} yields

\begin{lemma}
\label{Nlemma} If $\mathbf{E}\left[ X\right] =-a<0$ and condition (\ref%
{remainder}) is valid then \newline
(i) there exists $\delta _{0}\in (0,1/4)$ such that for $an/2-u\geq ~M$ and
all $\delta \in (0,\delta _{0})$ and $k\in \mathbb{Z}$
\begin{equation*}
\mathbf{P}_{u}(\max_{1\leq j\leq n}X_{j}\leq \delta an,\ S_{n}\geq k)\leq
\varepsilon _{M}(k)n^{-\beta -1},
\end{equation*}%
$\text{where }\varepsilon _{M}\left( k\right) \downarrow _{M\rightarrow
\infty }0$. Moreover, for any fixed $l$ and $\delta \in (0,1)$
\begin{equation*}
\lim_{J\rightarrow \infty }\limsup_{n\rightarrow \infty }b_{n}^{-1}\mathbf{P}%
\left( L_{n}\geq -N,\,\max_{J\leq j\leq n}X_{j}\geq \delta an,\,S_{n}\in
\lbrack l,l+1)\right) =0.
\end{equation*}%
(ii) for any fixed $\delta \in (0,1)$ and $K\geq 0$,%
\begin{equation*}
\lim_{M\rightarrow \infty }\limsup_{n\rightarrow \infty }b_{n}^{-1}\mathbf{P}%
\left( \delta an\leq X_{1}\leq an-M\sqrt{n}\ \text{or}\ X_{1}\geq an+M\sqrt{n%
};\left\vert S_{n}\right\vert \leq K\right) =0.
\end{equation*}%
(iii) for each fixed $\delta >0$ and $J\geq 2$
\begin{equation*}
\lim_{n\rightarrow \infty }b_{n}^{-1}\mathbf{P}\left( \cup _{i\neq
j}^{J}\left\{ X_{i}\geq \delta an,X_{j}\geq \delta an\right\} \right) =0.
\end{equation*}
\end{lemma}

Combining the limit for $J\rightarrow \infty$ in (i) with (iii), we get that
for any fixed $N,K \geq 0$
\begin{equation}  \label{twojumps}
\lim_{n\rightarrow \infty }b_{n}^{-1}\mathbf{P}\left( \cup _{i\neq
j}^{n}\left\{ X_{i}\geq \delta an,X_{j}\geq \delta an\right\} ; L_n\geq -N,
\vert S_n \vert \leq K \right) =0.
\end{equation}

\section{Proofs}

In this section we use the notation%
\begin{equation*}
\mathbf{E}_{\mathfrak{e}}\left[ \cdot \right] =\mathbf{E}\left[ \cdot \,|\,%
\mathcal{E}\right] ,\quad \mathbf{P}_{\mathfrak{e}}\left( \cdot \right) =%
\mathbf{P}\left( \cdot \,|\,\mathcal{E}\right)
\end{equation*}%
i.e., consider the expectation and probability given the environment $%
\mathcal{E}$. Our aim is to prove the following statement.

\begin{lemma}
\label{L_Extin}If Hypotheses A and B are valid
then there exists a constant $C_{0}>0$ such that, as $n\rightarrow \infty $
\begin{equation}
\mathbb{P}\left( Z_{n}>0\right) \sim C_{0}m^{n}\beta \frac{\mathbf{P}\left(
X>an\right) }{an}=C_{0}m^{n}b_{n}.  \label{AsTotal1}
\end{equation}
\end{lemma}

We recall from the discussion in Preliminaries that Hypotheses A and B
(or~B') ensure that there exists $g(\lambda )$ a.s. continuous on $[0,\infty
)$ with\textbf{\ }$g(0)=1$ and with $\mathbf{E}\left[ g(\lambda )\right]
<1,\lambda >0,$ such that for every continuous bounded function $H$ on~$%
[0,1] $%
\begin{equation}
\lim_{y\rightarrow \infty }\sup_{\lambda \geq 0}\left\vert \mathbf{E}\left[
H(f(e^{-\lambda /y}))\,|\,\,f^{\prime }(1)=y\right] -\mathbf{E}^{\ast }\left[
H(g(\lambda ))\right] \right\vert =0.  \label{As_f11}
\end{equation}

Making the change of measure in accordance with (\ref{Cmes1}) and (\ref%
{Cemes2}) we see that it is necessary to show that, as $n\rightarrow \infty $
\begin{equation}
\mathbf{E}\left[ \mathbf{P}_{\mathfrak{e}}\left( Z_{n}>0\right) e^{-\rho
S_{n}}\right] \sim C_{0}b_{n}.  \label{AfterChange}
\end{equation}

The proof of this fact is conducted into several steps which we split into
subsections.

\subsection{Time of the minimum of $S$}

First, we prove that the contribution to $\mathbf{E}\left[ \mathbf{P}_{%
\mathfrak{e}}\left( Z_{n}>0\right) e^{-\rho S_{n}}\right] $ may be of order $%
b_{n}$ only if the minimal value of $S$ within the interval $[0,n]$ is
attained at the beginning or at the end of this interval. To this aim we
use, as earlier, the notation $\tau _{n}=\min \left\{ 0\leq k\leq
n:S_{k}=L_{n}\right\} $ and show that the following statement is valid.

\begin{lemma}
\label{L_Begend}Given Hypotheses A and B we have%
\begin{equation*}
\lim_{M\rightarrow \infty }\lim_{n\rightarrow \infty }b_{n}^{-1}\mathbf{E}%
\left[ \mathbf{P}_{\mathfrak{e}}\left( Z_{n}>0\right) e^{-\rho S_{n}};\tau
_{n}\in \left[ M,n-M\right] \right] =0.
\end{equation*}
\end{lemma}

\begin{proof}
In view of the estimate%
\begin{equation*}
\mathbf{P}_{\mathfrak{e}}\left( Z_{n}>0\right) \leq \min_{0\leq k\leq n}%
\mathbf{P}_{\mathfrak{e}}\left( Z_{n}>0\right) \leq \exp \left\{ \min_{0\leq
k\leq n}S_{k}\right\} =e^{S_{\tau _{n}}}
\end{equation*}%
we have%
\begin{eqnarray}
&&\mathbf{E}\left[ \mathbf{P}_{\mathfrak{e}}\left( Z_{n}>0\right) e^{-\rho
S_{n}};\tau _{n}\in \left[ M,n-M\right] \right]  \notag \\
&&\quad \leq \mathbf{E}\left[ e^{S_{\tau _{n}}-S_{n}};\tau _{n}\in \left[
M,n-M\right] \right]  \notag \\
&&\quad =\sum_{k=M}^{n-M}\mathbf{E}\left[ e^{(1-\rho )S_{k}+\rho
(S_{k}-S_{n})};\tau _{n}=k\right]  \notag \\
&&\quad =\sum_{k=M}^{n-M}\mathbf{E}\left[ e^{(1-\rho )S_{k}};\tau _{k}=k%
\right] \mathbf{E}\left[ e^{-\rho S_{n-k}};L_{n-k}\geq 0\right] .
\label{rwAppr}
\end{eqnarray}%
Hence, using Lemma \ref{LExponent} we get%
\begin{eqnarray}
&&\mathbf{E}\left[ \mathbf{P}_{\mathfrak{e}}\left( Z_{n}>0\right) e^{-\rho
S_{n}};\tau _{n}\in \left[ M,n-M\right] \right]  \notag \\
&&\quad \leq \left( \sum_{k=M}^{\left[ n/2\right] }+\sum_{k=\left[ n/2\right]
+1}^{n-M}\right) \mathbf{E}\left[ e^{(1-\rho )S_{k}};\tau _{k}=k\right]
\mathbf{E}\left[ e^{-\rho S_{n-k}};L_{n-k}\geq 0\right]  \notag \\
&&\quad \leq \frac{C_{1}}{n}\mathbf{P}\left( X>\frac{an}{2}\right)
\sum_{k=M}^{\left[ n/2\right] }\mathbf{E}\left[ e^{(1-\rho )S_{k}};\tau
_{k}=k\right]  \notag \\
&&\quad +\frac{C_{2}}{n}\mathbf{P}\left( X>\frac{an}{2}\right) \sum_{k=M}^{%
\left[ n/2\right] }\mathbf{E}\left[ e^{-\rho S_{k}};L_{k}\geq 0\right] \leq
\varepsilon _{M}b_{n}  \label{Rem3}
\end{eqnarray}%
where $\varepsilon _{M}\rightarrow 0$ as $M\rightarrow \infty $.
\end{proof}

The following statement easily follows from (\ref{Rem3}) by taking $M=0$.

\begin{corollary}
\label{C_bound}Given Hypotheses A and B there exists $C\in \left( 0,\infty
\right) $ such that, for all $n=1,2,...$%
\begin{equation*}
\mathbf{E}\left[ \mathbf{P}_{\mathfrak{e}}\left( Z_{n}>0\right) e^{-\rho
S_{n}}\right] \leq \mathbf{E}\left[ e^{S_{\tau _{n}}-\rho S_{n}}\right] \leq
Cb_{n}.
\end{equation*}
\end{corollary}

\subsection{Fluctuations of the random walk $S$}

Introduce the event%
\begin{equation*}
\mathcal{C}_{N}=\left\{ -N<S_{\tau _{n}}\leq S_{n}\leq N+S_{\tau
_{n}}<N\right\} .
\end{equation*}%
In particular, given $\mathcal{C}_{N}$
\begin{equation*}
-N<S_{n}<N.
\end{equation*}%
In what follows we agree to denote by $\varepsilon _{N},\varepsilon _{N,n}$
or $\varepsilon _{N,K,n}$ functions of the low indices such that
\begin{equation*}
\lim_{N\rightarrow \infty }\varepsilon _{N}=\lim_{N\rightarrow \infty
}\limsup_{n\rightarrow \infty }\left\vert \varepsilon _{N,n}\right\vert
=\lim_{N\rightarrow \infty }\limsup_{K\rightarrow \infty
}\limsup_{n\rightarrow \infty }\left\vert \varepsilon _{N,K,n}\right\vert =0,
\end{equation*}%
i.e., the $\limsup $ (or \thinspace $\lim $) are sequentially taken with
respect to the indices of $\varepsilon _{\cdot \cdot \cdot }$ in the reverse
order. Note that the functions are not necessarily the same in different
formulas or even within one and the same complicated expression.

\begin{lemma}
\label{L_SmallEnd}Given Hypotheses A and B for any fixed $k$%
\begin{equation*}
\lim_{N\rightarrow \infty }\limsup_{n\rightarrow \infty }b_{n}^{-1}\mathbf{E}%
\left[ \mathbf{P}_{\mathfrak{e}}\left( Z_{n}>0\right) e^{-\rho S_{n}};\tau
_{n}=k,\mathcal{\bar{C}}_{N}\right] =0
\end{equation*}%
and%
\begin{equation*}
\lim_{N\rightarrow \infty }\limsup_{n\rightarrow \infty }b_{n}^{-1}\mathbf{E}%
\left[ \mathbf{P}_{\mathfrak{e}}\left( Z_{n}>0\right) e^{-\rho S_{n}};\tau
_{n}=n-k,\mathcal{\bar{C}}_{N}\right] =0.
\end{equation*}
\end{lemma}

\begin{proof}
In view of (\ref{Hirano2})%
\begin{eqnarray*}
&&\mathbf{E}\left[ \mathbf{P}_{\mathfrak{e}}\left( Z_{n}>0\right) e^{-\rho
S_{n}};\tau _{n}=k,S_{n}-S_{\tau _{n}}\geq N\right] \\
&&\quad \leq \mathbf{E}\left[ e^{(1-\rho )S_{\tau _{n}}}e^{-\rho
(S_{n}-S_{\tau _{n}})};\tau _{n}=k,S_{n}-S_{\tau _{n}}\geq N\right] \\
&&\qquad \leq \mathbf{E}\left[ e^{-\rho S_{n-k}};L_{n-k}\geq 0,S_{n-k}\geq N%
\right] \leq \varepsilon _{N}b_{n}
\end{eqnarray*}%
where $\varepsilon _{N}\rightarrow 0$ as $N\rightarrow \infty $ since $%
\int_{0}^{\infty }\exp (-\rho z)V(-z)dz<\infty $. Further,
\begin{eqnarray}
&&\mathbf{E}\left[ \mathbf{P}_{\mathfrak{e}}\left( Z_{n}>0\right) e^{-\rho
S_{n}};\tau _{n}=k,S_{\tau _{n}}\leq -N\right]  \notag \\
&&\quad \leq \mathbf{E}\left[ e^{(1-\rho )S_{\tau _{n}}}e^{-\rho
(S_{n}-S_{\tau _{n}})};\tau _{n}=k,S_{\tau _{n}}\leq -N\right]  \notag \\
&&\qquad \leq e^{-(1-\rho )N}\mathbf{E}\left[ e^{-\rho S_{n-k}};L_{n-k}\geq 0%
\right] \leq \varepsilon _{N}b_{n}.  \label{Nmin}
\end{eqnarray}%
This, in particular, means that%
\begin{equation}
\mathbf{E}\left[ e^{(1-\rho )S_{\tau _{n}}}e^{-\rho (S_{n}-S_{\tau
_{n}})};\tau _{n}=k,S_{n}\notin (-N,N)\right] =\varepsilon _{N,n}b_{n}
\label{Neg}
\end{equation}%
and%
\begin{eqnarray*}
&&\mathbf{E}\left[ \mathbf{P}_{\mathfrak{e}}\left( Z_{n}>0\right) e^{-\rho
S_{n}};\tau _{n}=k\right] \\
&&\quad =\mathbf{E}\left[ \mathbf{P}_{\mathfrak{e}}\left( Z_{n}>0\right)
e^{-\rho S_{n}};\tau _{n}=k,S_{\tau _{n}}\geq -N,S_{n}-S_{\tau _{n}}\leq N%
\right] +\varepsilon _{N,n}b_{n}.
\end{eqnarray*}%
Similarly, by (\ref{Hirano4})%
\begin{eqnarray*}
&&\mathbf{E}\left[ \mathbf{P}_{\mathfrak{e}}\left( Z_{n}>0\right) e^{-\rho
S_{n}};\tau _{n}=n-k,S_{\tau _{n}}\leq -N\right] \\
&&\quad \leq \mathbf{E}\left[ e^{(1-\rho )S_{\tau _{n}}}e^{-\rho
(S_{n}-S_{\tau _{n}})};\tau _{n}=n-k,S_{\tau _{n}}\leq -N\right] \\
&&\qquad \leq \mathbf{E}\left[ e^{(1-\rho )S_{n-k}};\tau
_{n-k}=n-k,S_{n-k}\leq -N\right] \\
&&\qquad =\mathbf{E}\left[ e^{(1-\rho )S_{n-k}};M_{n-k}<0,S_{n-k}\leq -N%
\right] =\varepsilon _{N,n}b_{n}
\end{eqnarray*}%
and
\begin{eqnarray*}
&&\mathbf{E}\left[ \mathbf{P}_{\mathfrak{e}}\left( Z_{n}>0\right) e^{-\rho
S_{n}};\tau _{n}=n-k,S_{n}-S_{\tau _{n}}\geq N\right] \\
&&\quad \leq \mathbf{E}\left[ e^{(1-\rho )S_{\tau _{n}}}e^{-\rho
(S_{n}-S_{\tau _{n}})};\tau _{n}=n-k,S_{n}-S_{\tau _{n}}\geq N\right] \\
&&\qquad \leq e^{-\rho N}\mathbf{E}\left[ e^{(1-\rho )S_{n-k}};\tau
_{n-k}=n-k\right] \\
&&\qquad =e^{-\rho N}\mathbf{E}\left[ e^{(1-\rho )S_{n-k}};M_{n-k}<0\right]
=\varepsilon _{N,n}b_{n}.
\end{eqnarray*}%
As a result we get%
\begin{eqnarray*}
&&\mathbf{E}\left[ \mathbf{P}_{\mathfrak{e}}\left( Z_{n}>0\right) e^{-\rho
S_{n}};\tau _{n}=n-k\right] \\
&&\quad =\mathbf{E}\left[ \mathbf{P}_{\mathfrak{e}}\left( Z_{n}>0\right)
e^{-\rho S_{n}};\tau _{n}=n-k,S_{\tau _{n}}\geq -N,S_{n}-S_{\tau _{n}}\leq N%
\right] +\varepsilon _{N,n}b_{n}.
\end{eqnarray*}%
This completes the proof of the lemma.
\end{proof}

Lemmas \ref{L_Begend} and \ref{L_SmallEnd} easily imply the following
statement:

\begin{corollary}
\label{C_TimeSpace}Under Hypotheses A and B%
\begin{eqnarray}
&&\mathbf{E}\left[ \mathbf{P}_{\mathfrak{e}}\left( Z_{n}>0\right) e^{-\rho
S_{n}}\right]  \notag \\
&&\quad =\mathbf{E}\left[ \mathbf{P}_{\mathfrak{e}}\left( Z_{n}>0\right)
e^{-\rho S_{n}};\left\vert S_{n}\right\vert <N;\tau _{n}\in \lbrack 0,M]\cup
\lbrack n-M,n]\right] +\varepsilon _{N,M,n}b_{n}  \notag \\
&&\qquad =\mathbf{E}\left[ \mathbf{P}_{\mathfrak{e}}\left( Z_{n}>0\right)
e^{-\rho S_{n}};\left\vert S_{n}\right\vert <N\right] +\varepsilon
_{N,n}b_{n}  \notag \\
&&\quad \qquad =\mathbf{E}\left[ \mathbf{P}_{\mathfrak{e}}\left(
Z_{n}>0\right) e^{-\rho S_{n}};S_{\tau _{n}}\geq -N,S_{n}<N\right] +\tilde{%
\varepsilon}_{N,n}b_{n}  \label{Remainder1}
\end{eqnarray}%
where
\begin{equation*}
\lim_{N\rightarrow \infty }\limsup_{M\rightarrow \infty
}\limsup_{n\rightarrow \infty }\left\vert \varepsilon _{N,M,n}\right\vert
=\lim_{N\rightarrow \infty }\limsup_{n\rightarrow \infty }\left( \left\vert
\varepsilon _{N,n}\right\vert +\left\vert \tilde{\varepsilon}%
_{N,n}\right\vert \right) =0.
\end{equation*}
\end{corollary}

\subsection{Asymptotic of the survival probability}

In this section we investigate in detail the properties of the survival
probability for the processes meeting Hypotheses A and B. As we know (see (%
\ref{Cemes2})) this probability is expressed as
\begin{equation*}
\mathbb{P}\left( Z_{n}>0\right) =m^{n}\mathbf{E}\left[ \mathbf{P}_{\mathfrak{%
e}}\left( Z_{n}>0\right) e^{-\rho S_{n}}\right] .
\end{equation*}

We wish to show that $\mathbf{E}\left[ \mathbf{P}_{\mathfrak{e}}\left(
Z_{n}>0\right) e^{-\rho S_{n}}\right] $ is of order $b_{n}$ as $n\rightarrow
\infty $.

First we get rid of the trajectories giving the contribution of the order $%
o(b_{n})$ to the quantity in question. Let%
\begin{equation*}
\mathcal{D}_{N}(j)=\left\{ -N<S_{\tau _{n}}\leq S_{n}<N,\,X_{j}\geq \delta
an\right\} .
\end{equation*}

\begin{lemma}
\label{L_Negl2}If Hypotheses A and B are valid then there exists $\,\delta
\in \left( 0,1/4\right) $ such that%
\begin{equation}
\mathbf{E}\left[ \mathbf{P}_{\mathfrak{e}}\left( Z_{n}>0\right) \exp (-\rho
S_{n})\right] =\sum_{j=1}^{J}\mathbf{E}\left[ \mathbf{P}_{\mathfrak{e}%
}\left( Z_{n}>0\right) \exp (-\rho S_{n});\mathcal{D}_{N}(j)\right]
+\varepsilon _{N,J,n}b_{n}.  \label{MorNegl}
\end{equation}
\end{lemma}

\begin{proof}
In view of Corollary \ref{C_TimeSpace}, we just need to prove that
\begin{eqnarray}
&&\mathbf{E}\left[ \mathbf{P}_{\mathfrak{e}}\left( Z_{n}>0\right) e^{-\rho
S_{n}};S_{\tau _{n}}\geq -N,S_{n}<N\right]  \notag  \label{toprove} \\
&&\qquad =\sum_{j=1}^{J}\mathbf{E}\left[ \mathbf{P}_{\mathfrak{e}}\left(
Z_{n}>0\right) \exp (-\rho S_{n});\mathcal{D}_{N}(j)\right] +\varepsilon
_{N,J,n}b_{n}.
\end{eqnarray}%
From the estimate
\begin{equation}
\mathbf{P}_{\mathfrak{e}}\left( Z_{n}>0\right) \exp (-\rho S_{n})\leq \exp
(S_{\tau _{n}}-\rho S_{n})=\exp ((1-\rho )S_{\tau _{n}}-\rho (S_{n}-S_{\tau
_{n}}))\leq 1  \label{Le1}
\end{equation}%
we deduce by Lemma \ref{Nlemma} (i) that
\begin{equation*}
\mathbf{E}\left[ \mathbf{P}_{\mathfrak{e}}\left( Z_{n}>0\right) e^{-\rho
S_{n}};S_{\tau _{n}}\geq -N,S_{n}<N,\max_{0\leq j\leq n}X_{j}<\delta an%
\right] =\varepsilon _{N,n}b_{n}
\end{equation*}%
for $\delta \in (0,\delta _{0})$ and
\begin{equation*}
\mathbf{E}\left[ \mathbf{P}_{\mathfrak{e}}\left( Z_{n}>0\right) e^{-\rho
S_{n}};S_{\tau _{n}}\geq -N,S_{n}<N,\max_{J\leq j\leq n}X_{j}\geq \delta an%
\right] =\varepsilon _{N,J,n}b_{n}.
\end{equation*}%
Thus,
\begin{eqnarray*}
&&\mathbf{E}\left[ \mathbf{P}_{\mathfrak{e}}\left( Z_{n}>0\right) e^{-\rho
S_{n}};S_{\tau _{n}}\geq -N,S_{n}<N\right] \\
&&\quad =\mathbf{E}\left[ \mathbf{P}_{\mathfrak{e}}\left( Z_{n}>0\right)
e^{-\rho S_{n}};S_{\tau _{n}}\geq -N,S_{n}<N,\max_{0\leq j\leq J}X_{j}\geq
\delta an\right] +\varepsilon _{N,J,n}b_{n}.
\end{eqnarray*}%
Finally thanks to Lemma \ref{Nlemma}(iii), there is only one big jump
(before $J$),i.e.
\begin{equation*}
\mathbf{E}\left[ \mathbf{P}_{\mathfrak{e}}\left( Z_{n}>0\right) e^{-\rho
S_{n}};S_{\tau _{n}}\geq -N,S_{n}<N,\cup _{i\neq j}^{J}\left\{ X_{i}\geq
\delta an,X_{j}\geq \delta an\right\} \right] =\varepsilon _{N,J,n}b_{n}.
\end{equation*}%
It yields $(\ref{toprove})$ and ends up the proof.
\end{proof}

Now we fix $j\in \lbrack 1,J]$ and investigate the quantity%
\begin{equation*}
\mathbf{E}\left[ \mathbf{P}_{\mathfrak{e}}\left( Z_{n}>0\right) \exp (-\rho
S_{n});\mathcal{D}_{N}(j)\right] .
\end{equation*}

First, we check that $S_{j-1}$ is bounded on the event we focus on.

\begin{lemma}
\label{L_Sj}If Hypotheses A and B are valid then, for every fixed $j$
\begin{equation*}
\mathbf{E}\left[ \mathbf{P}_{\mathfrak{e}}\left( Z_{n}>0\right) \exp (-\rho
S_{n});\left\vert S_{j-1}\right\vert \geq N,X_{j}\geq \delta an\right]
=\varepsilon _{N,n}b_{n}.
\end{equation*}
\end{lemma}

\begin{proof}
First observe that
\begin{eqnarray*}
&&\mathbf{E}\left[ \mathbf{P}_{\mathfrak{e}}\left( Z_{n}>0\right) \exp
(-\rho S_{n});S_{j-1}\leq -N,X_{j}\geq \delta an\right] \\
&&\quad\leq\mathbf{E}\left[ \exp ((1-\rho )S_{\tau _{n}}-\rho (S_{n}-S_{\tau
_{n}}));S_{j-1}\leq -N,X_{j}\geq \delta an\right] \\
&&\quad\leq\mathbf{E}\left[ \exp ((1-\rho )S_{\tau _{n}});S_{j-1}\leq
-N,X_{j}\geq \delta an\right] \\
&&\quad\leq\mathbf{E}\left[ \exp (-(1-\rho )N);X_{j}\geq \delta an\right] \\
&&\quad=\exp (-(1-\rho )N)\mathbf{P}(X\geq \delta an)=\varepsilon
_{N,n}b_{n}.
\end{eqnarray*}%
Further, taking $\gamma \in (0,1)$ such that $\gamma \beta >1$, we get%
\begin{eqnarray}
&&\mathbf{E}\left[ \exp (S_{\tau _{n}}-\rho S_{n});S_{j-1}\geq n^{\gamma
},X_{j}\geq \delta an\right] \leq \mathbf{P}(S_{j-1}\geq n^{\gamma })\mathbf{%
P}(X\geq \delta an)  \notag \\
&&\quad \leq j\mathbf{P}(X\geq n^{\gamma }/j)\mathbf{P}(X\geq \delta an)\sim
\frac{j^{\beta +1}}{n^{\gamma \beta }}l(n^{\gamma })\mathbf{P}(X\geq \delta
an)=\varepsilon _{n}b_{n}.  \label{Djlarge}
\end{eqnarray}%
Consider now the situation $S_{j-1}\in \lbrack N,n^{\gamma }],\,j\geq 2$ and
write
\begin{eqnarray*}
&&\mathbf{E}\left[ \exp (S_{\tau _{n}}-\rho S_{n});S_{j-1}\in \lbrack
N,n^{\gamma }],X_{j}\geq \delta an\right] \\
&&\qquad =\int_{N}^{n^{\gamma }}\int_{-\infty }^{0}\mathbf{P}(S_{j-1}\in
dy,L_{j-1}\in dz)H_{n,\delta }(y,z),
\end{eqnarray*}%
where
\begin{eqnarray*}
H_{n,\delta }(y,z) &=&\int_{\delta an}^{\infty }\mathbf{P}(X\in
dt)\int_{-\infty }^{0}\int_{v}^{\infty }\mathbf{P}_{y+t}(L_{n-j}\in
dv,S_{n-j}\in dw)e^{z\wedge v}e^{-\rho w} \\
&=&\int_{\delta an+y}^{\infty }\mathbf{P}(X\in dt-y)\int_{-\infty
}^{0}\int_{v}^{\infty }\mathbf{P}_{t}(L_{n-j}\in dv,S_{n-j}\in dw)e^{z\wedge
v}e^{-\rho w}.
\end{eqnarray*}%
By our conditions $\ \mathbf{P}(X\in dt-y)=\mathbf{P}(X\in dt)\left(
1+o\left( 1\right) \right) $ \ uniformly in $y\in \lbrack 0,n^{\gamma }]$
and $t\geq \delta an$. Thus, for all sufficiently large $n$%
\begin{eqnarray*}
H_{n,\delta }(y,z) &\leq &2\int_{\delta an}^{\infty }\mathbf{P}(X\in
dt)\int_{-\infty }^{0}\int_{v}^{\infty }\mathbf{P}_{t}(L_{n-j}\in
dv,S_{n-j}\in dw)e^{z\wedge v}e^{-\rho w} \\
&\leq &2\int_{\delta an}^{\infty }\mathbf{P}(X\in dt)\int_{-\infty
}^{0}\int_{v}^{\infty }\mathbf{P}_{t}(L_{n-j}\in dv,S_{n-j}\in
dw)e^{v}e^{-\rho w} \\
&=&2\int_{\delta an}^{\infty }\mathbf{P}(X\in dt)\mathbf{E}_{t}\left[
e^{S_{\tau _{n-j}}-\rho S_{n-j}}\right] \\
&\leq &2\mathbf{E}_{0}\left[ e^{S_{\tau _{n-j+1}}-\rho S_{n-j+1}};X_{1}\geq
\delta an\right] =2H_{n,\delta }(0,0).
\end{eqnarray*}%
By integrating this inequality we get for sufficiently large $n$
\begin{eqnarray*}
&&\int_{N}^{n^{\gamma }}\int_{-\infty }^{0}\mathbf{P}(S_{j-1}\in
dy,L_{j-1}\in dz)H_{n,\delta }(y,z) \\
&&\qquad \qquad \leq 2\int_{N}^{n^{\gamma }}\int_{-\infty }^{0}\mathbf{P}%
(S_{j-1}\in dy,L_{j-1}\in dz)H_{n,\delta }(0,0) \\
&&\qquad \qquad \leq 2\mathbf{P}(S_{j-1}\geq N)\mathbf{E}_{0}\left[
e^{S_{\tau _{n-j+1}}-\rho S_{n-j+1}};X_{1}\geq \delta an\right] .
\end{eqnarray*}%
Since
\begin{eqnarray*}
b_{n}^{-1}H_{n,\delta }(0,0) &=&b_{n}\mathbf{E}\left[ e^{S_{\tau
_{n-j+1}}-\rho S_{n-j+1}};X_{1}\geq \delta an\right] \\
&\leq &b_{n}^{-1}\mathbf{E}\left[ e^{S_{\tau _{n-j+1}}-\rho S_{n-j+1}}\right]
=O(1)
\end{eqnarray*}%
as $n\rightarrow \infty $ (see Corollary \ref{C_bound}) and $\mathbf{P}%
(S_{j-1}\geq N)\rightarrow 0$ as $N\rightarrow \infty ,$ we obtain%
\begin{equation}
\mathbf{E}\left[ \exp (S_{\tau _{n}}-\rho S_{n});\,S_{j-1}\in \lbrack
N,n^{\gamma }],X_{j}\geq \delta n\right] =\varepsilon _{N,n}b_{n}.
\label{LargeSj}
\end{equation}

Combining (\ref{Djlarge}) and (\ref{LargeSj}) proves the lemma.
\end{proof}

\begin{lemma}
\label{L_remF2}Given Hypotheses A and B we have for each fixed $j$%
\begin{equation*}
\mathbf{E}\left[ \mathbf{P}_{\mathfrak{e}}\left( Z_{n}>0\right) \exp (-\rho
S_{n});\left\vert S_{n}-S_{j-1}\right\vert >K,X_{j}\geq \delta an\right]
=\varepsilon _{K,n}(j)b_{n}.
\end{equation*}
\end{lemma}

\begin{proof}
We know from Lemma \ref{L_Sj} that only the values $S_{j-1}\leq N$ for
sufficiently large but fixed~$N$ are of importance. Thus, we just need to
prove that for fixed~$N$
\begin{equation*}
\mathbf{E}\left[ e^{S_{\tau _{n}}-\rho S_{n}};S_{j-1}\leq N,\left\vert
S_{n}-S_{j-1}\right\vert >K,X_{j}\geq \delta an\right] =\varepsilon
_{N,K,n}(j)b_{n}
\end{equation*}%
where $\lim_{K\rightarrow \infty }\limsup_{n\rightarrow \infty }\left\vert
\varepsilon _{N,K,n}(j)\right\vert =0$. To this aim we set $L_{j,n}=\min
\{S_{k}-S_{j-1}:j-1\leq k\leq n\}$ and, using the inequality $S_{\tau
_{n}}\leq S_{j-1}+L_{j,n},$ deduce the estimate
\begin{eqnarray*}
&&\mathbf{E}\left[ e^{S_{\tau _{n}}-\rho S_{n}};S_{j-1}\leq N,\left\vert
S_{n}-S_{j-1}\right\vert >K,X_{j}\geq \delta an\right] \\
&&\quad \leq \mathbf{E}\left[ e^{S_{j-1}+L_{j,n}-\rho (S_{n}-S_{j-1})-\rho
S_{j-1}};S_{j-1}\leq N,\left\vert S_{n}-S_{j-1}\right\vert >K\right] \\
&&\quad =\mathbf{E}\left[ e^{(1-\rho )S_{j-1}};S_{j-1}\leq N\right] \mathbf{E%
}\left[ e^{L_{j,n}-\rho (S_{n}-S_{j-1})};\left\vert S_{n}-S_{j-1}\right\vert
>K\right] .
\end{eqnarray*}%
We conclude with $\mathbf{E}\left[ e^{(1-\rho )S_{j-1}};S_{j-1}\leq N\right]
<\infty $ and we can now control the term
\begin{equation*}
\mathbf{E}\left[ e^{L_{j,n}-\rho (S_{n}-S_{j-1})};\left\vert
S_{n}-S_{j-1}\right\vert >K\right] =\mathbf{E}\left[ e^{S_{\tau
_{n-j+1}}-\rho S_{n-j+1}};\left\vert S_{n-j+1}\right\vert >K\right]
\end{equation*}%
by $\varepsilon _{K,n}b_{n}$. Indeed it is now exactly the term controlled
in a similar situation in~(\ref{Neg}).
\end{proof}

We give the last technical lemma.

\begin{lemma}
\label{techn} Assume that $g$ is a random function which satisfies (\ref%
{As_f11}). Then for every (deterministic) probability generating function $%
h\left( s\right) $ and every $\varepsilon >0$ there exists $\kappa >0$ such
that
\begin{equation*}
\left\vert \mathbf{E}\left[ 1-h(g(e^{v}w))\right] -\mathbf{E}\left[
1-h(g(e^{v^{\prime }}w))\right] \right\vert \leq h^{\prime }(1)\varepsilon
\end{equation*}%
for $|v-v^{\prime }|\leq \kappa ,w\in \lbrack 0,2].$
\end{lemma}

\begin{proof}
Clearly,
\begin{equation*}
\left\vert \mathbf{E}\left[ 1-h(g(e^{v}w))\right] -\mathbf{E}\left[
1-h(g(e^{v^{\prime }}w))\right] \right\vert \leq h^{\prime }(1)\mathbf{E}%
\left[ |g(e^{v^{\prime }}w)-g(e^{v}w)|\right] .
\end{equation*}%
We know that $0\leq g(\lambda )\leq 1$ for all $\lambda \in \lbrack 0,\infty
),$ $g(\lambda )$ is nonincreasing with respect to $\lambda $ a.s. and has a
finite limit as $\lambda \rightarrow \infty $. Therefore, $g(\lambda )$ is
a.s. uniformly continuous on $[0,\infty )$ implying that a.s.
\begin{equation*}
\lim_{\kappa \rightarrow 0}\sup_{|v-v^{\prime }|\leq \kappa ,w\in \lbrack
0,2]}|g(e^{v^{\prime }}w)-g(e^{v}w)|=0.
\end{equation*}%
Hence, by the bounded convergence theorem
\begin{equation*}
\sup_{|v-v^{\prime }|\leq \kappa ,w\in \lbrack 0,2]}\mathbf{E}\left[
|g(e^{v^{\prime }}w)-g(e^{v}w)|\right] \leq \mathbf{E}\left[
\sup_{|v-v^{\prime }|\leq \kappa ,w\in \lbrack 0,2]}|g(e^{v^{\prime
}}w)-g(e^{v}w)|\right]
\end{equation*}%
goes to zero as $\kappa \rightarrow 0$, which ends up the proof.
\end{proof}

$\newline
$

Let $\sigma ^{2}=VarX$ , $S_{n,j}=S_{n}-S_{j},\ 0\leq j\leq n,$ and
\begin{equation*}
G_{n,j}=-\frac{S_{n,j}+an}{\sigma \sqrt{n}}.
\end{equation*}

Using the notation (\ref{NotIterat}), we write
\begin{equation*}
\mathbf{P}_{\mathfrak{e}}\left( Z_{n}>0\right) =1-f_{0,n}(0)
\end{equation*}
put $\mathbf{X}_{j,n}=\left( X_{j+1},\cdots ,X_{n}\right),$ $\mathbf{X}%
_{n,j}=\left( X_{n},\cdots ,X_{j+1}\right),$ and set
\begin{equation*}
Y_{j}=F(S_{0},\mathbf{S}_{0,j-1}), \qquad Y_{j,n}=F_{n}(S_{n}-S_{j-1},
\mathbf{X}_{j,n}),\qquad Y_{n,j}=F_{n}(S_{n}-S_{j-1},\mathbf{X}_{n,j}),
\end{equation*}%
where $F,F_{n}$ are positive equibounded measurable functions. \newline

Since $f_{j,n}$ is distributed as $f_{n,j},$ we write
\begin{align*}
& \mathbf{E}[Y_{j}Y_{j,n}\mathbf{P}_{\mathfrak{e}}\left( Z_{n}>0\right)
e^{-\rho S_{n}};X_{j}\geq \delta an] \\
& \quad =\mathbf{E}[Y_{j}Y_{j,n}\left( 1-f_{0,n}(0)\right) e^{-\rho
S_{n}};X_{j}\geq \delta an] \\
& \quad =\mathbf{E}[Y_{j}Y_{n,j}(1-f_{0,j-1}(f_{j}(f_{n,j}(0))))e^{-\rho
S_{n}};X_{j}\geq \delta an] \\
& \quad =\mathbf{E}[Y_{j}e^{-\rho S_{j-1}}Y_{n,j}\left(
1-f_{0,j-1}(f_{j}(f_{n,j}(0)))\right) e^{-\rho (S_{n}-S_{j-1})};X_{j}\geq
\delta an] \\
& \quad =\mathbf{E}[Y_{j}e^{-\rho S_{j-1}}Y_{n,j}\left(
1-f_{0,j-1}(f_{j}(1-e^{S_{n,j}}W_{n,j}))\right) e^{-\rho
S_{n,j-1}};X_{j}\geq \delta an]
\end{align*}%
where $W_{n,j}$ were defined in (\ref{DefWn}). Our aim is to obtain an
approximation to this expression. \newline
To simplify notation we let%
\begin{equation*}
\bar{h}\left( s\right) =1-h(s)
\end{equation*}%
for a probability generating function $h\left( s\right) $. %
For fixed positive $M$ and $K$, we set
\begin{equation*}
B_{j,n}=\{S_{n,j}\in \lbrack -K,K],\,X_{j}-na\in \lbrack -M\sqrt{n},M\sqrt{n}%
]\},
\end{equation*}%
and define
\begin{equation*}
F_{n,j}(h,K,M)=\mathbf{E}\left[ e^{-\rho (S_{n,j-1})}{Y}_{n,j}\bar{h}\left(
f_{j}\left( \exp \left\{ -e^{S_{n,j}}W_{n,j}\right\} \right) \right)
\,;B_{j,n}\right] .
\end{equation*}%
%
%
%
%
We now introduce a random function $g_{j}$ on the probability space $(\Omega
,\mathbf{P})$, whose distribution is specified by $\mathbf{P}^{\ast }$ (i.e.
$\mathbf{E}(H(g_{j}))=\mathbf{E}^{\ast }(H(g))$ for any bounded $H$).
Moreover we choose $g_{j}$ such that $g_{j}$ is independent of $(f_{k}:k\neq
j).$
As we
have mentioned, it is always possible by extending  the initial
probability space  if required.
We denote $Y_{n,j}(v)=F_{n}(v,\mathbf{X}_{n,j})$ and consider
\begin{equation*}
O_{n,j}(h,K,M)=\int_{-K}^{K}e^{-\rho v}dv\mathbf{E}\left[ Y_{n,j}(v)\bar{h}%
\left( g_{j}\left( e^{v}W_{n,j}\right) \right) \,;\sigma G_{n,j}\in \left[
-M,M\right] \right]
\end{equation*}%
where $g_{j}$ is independent of $(S_{k}:k\geq 0)$ and $(f_{k}:k\neq j)$.

\begin{lemma}
\label{limO} For all $K,M\geq 0$ and and any probability generating function
$h$ we have
\begin{equation*}
\lim_{n\rightarrow \infty }|b_{n}^{-1}F_{n,j}(h,K,M)-O_{n,j}(h,K,M)|=0
\end{equation*}
\end{lemma}

\begin{proof}
Let $\mathcal{F}_{j,n}$ be the $\sigma $-algebra generated by the random
variables
\begin{equation*}
\text{ }\left( f_{k},X_{k}\right), \quad k=1,2,...,j-1,j+1,...,n
\end{equation*}
and
\begin{equation*}
V(y,\mathbf{X}_{j,n})=e^{-\rho y}F_{n}\left( y;\mathbf{X}_{n,j}\right)
1_{\left\{ \left\vert y\right\vert \in \left[ -K,K\right] \right\}}.
\end{equation*}
Using the uniform convergence (\ref{densunif}), the change of variables $%
t=(x_j-an-M\sqrt{n})/\sqrt{n}$ ensures that,
\begin{eqnarray*}
&&b_{n}^{-1}F_{n,j}(h,K,M) \\
&&\qquad = b_{n}^{-1}\mathbf{E}\bigg[ \int_{an-M\sqrt{n}}^{an+M\sqrt{n}%
}V(S_{n,j}+x_{j},\mathbf{X}_{n,j}) \\
&& \qquad \qquad \qquad \qquad \times \mathbf{E}\left[ \bar{h}\left(
f_{j}\left( \exp \left\{ -e^{S_{n,j}}W_{n,j}\right\} \right) \right) |%
\mathcal{F}_{j,n};X_{j}=x_j\right] \,\mathbf{p}_{X_{j}}(x_{j})dx_{j}\bigg] \\
&& \qquad \sim \mathbf{E}\bigg[ \int_{an-M\sqrt{n}}^{an+M\sqrt{n}%
}V(S_{n,j}+x_{j},\mathbf{X}_{n,j}) \\
&& \qquad \qquad \qquad \qquad \times \mathbf{E}\left[ \bar{h}\left(
f_{j}\left( \exp \left\{ -e^{S_{n,j}}W_{n,j}\right\} \right) \right) |%
\mathcal{F}_{j,n};X_{j}=x_{j}\right] \,dx_{j}\bigg],
\end{eqnarray*}
when $n\rightarrow \infty$. Moreover, the uniform convergence in (\ref{Unif})
with respect to any compact set of $\lambda $ from $[0,\infty )$ ensures
that, uniformly for $\left\vert x-an\right\vert \leq Mn^{1/2}$, $w\in
\lbrack 0,2]$ and $\left\vert v\right\vert \leq K$ we have
\begin{equation*}
|\mathbf{E}\left[ \bar{h} \left( f_{j}\left( \exp \left( -e^{v}w\right)
\right) \right) \,|\,X_{j}=x\right] -\mathbf{E}\left[ \bar h (g_{j}(e^{v}w))%
\right] |\leq \varepsilon _{n}.
\end{equation*}
Denoting $\mathcal{F}_{j,n}^{\ast }$ the $\sigma $-algebra generated by the
random variables
\begin{equation*}
\text{ }X_{k},\qquad k=1,2,...,j-1,j+1,...,n
\end{equation*}
we get, as $n\rightarrow \infty$, with $\mathbf{x}_{n,j}=(x_n,\ldots,x_{j+1})$,
\begin{eqnarray*}
&&{b_{n}^{-1}F_{n,j}(h,K,M)} \\
& & \qquad \sim \mathbf{E}\left[ \int_{an-M\sqrt{n}}^{an+M\sqrt{n}%
}V(S_{n,j}+x_{j},\mathbf{X}_{n,j})\mathbf{E}\left[ \bar{h}%
(g_{j}(e^{S_{n,j}+x_{j}}W_{n,j}))|\mathcal{F}^{\ast}_{j,n}\right] \,dx_{j}%
\right] \\
&& \qquad = \mathbf{E}\left[ \int_{an-M\sqrt{n}}^{an+M\sqrt{n}%
}V(S_{n,j}+x_{j},\mathbf{X}_{n,j})\bar{h}%
(g_{j}(e^{S_{n,j}+x_{j}}W_{n,j}))dx_{j}\right] \\
&&\qquad \sim \int_{an-M\sqrt{n}}^{an+M\sqrt{n}}dx_{j}\int_{\left\vert
\mathbf{x}_{n,j-1}\right\vert \leq K}V(\left\vert \mathbf{x}%
_{n,j-1}\right\vert , \mathbf{x}_{n,j}) \\
&&\qquad \qquad \qquad \qquad \qquad\times \mathbf{E}\left[ \bar{h}%
(g_{j}(e^{\left\vert \mathbf{x}_{n,j-1}\right\vert }W_{n,j})|\mathbf{X}%
_{n,j}=\mathbf{x}_{n,j}\right] \,\prod_{i=j+1}^{n}\mathbf{p}_{X_{i}}\left(
x_{i}\right) d{x}_{i}.
\end{eqnarray*}%
Making the change of variables
\begin{equation*}
v=\left\vert \mathbf{x}_{n,j-1}\right\vert
=x_{n}+x_{n-1}+...+x_{j}; \qquad z_{i}=x_{i}, \quad i=j+1,...,n
\end{equation*}%
and setting%
\begin{equation*}
D_{n,j}\left( K,M\right) =\left\{ \left\vert v\right\vert \leq K,\left\vert
v-x_{j+1}-x_{j+2}-...-x_{n}+an\right\vert \leq M\sqrt{n}\right\},
\end{equation*}%
we arrive at%
\begin{eqnarray*}
&&{b_{n}^{-1}F_{n,j}(h,K,M)} \\
&&\quad \sim \int_{D_{n,j}\left( K,M\right) }e^{-\rho v}F_{n}(v,\mathbf{x}%
_{n,j})\mathbf{E}\left[ \bar{h}(g_{j}(e^{v}W_{n,j})|\mathbf{X}_{n,j}=\mathbf{%
x}_{n,j}\right] \,\prod_{i=j+1}^{n}\mathbf{p}_{X_{i}}\left( x_{i}\right) dx_{i}dv \\
&&\quad \sim \int_{\left\vert v\right\vert \leq K}e^{-\rho v}\mathbf{E}\left[
Y_{n,j}\left( v\right) \bar{h}(g_{j}(e^{v}W_{n,j});\sigma G_{n,j}\in \left[
-M,M\right] \right] dv.
\end{eqnarray*}%
It completes the proof.
\end{proof}

$\newline
$ Observe that by monotonicity
\begin{equation}
\lim_{n\rightarrow \infty }W_{n,j}=\lim_{n\rightarrow \infty }\frac{%
1-f_{n,j}\left( 0\right) }{e^{S_{n}-S_{j}}}=W_{j}\quad \text{a.s.}
\label{monot}
\end{equation}%
and $W_{j}\overset{d}{=}W,\,j=1,2,...$ where $\mathbf{P}(W\in (0,1])=1$ in
view of conditions (\ref{Mom}) and Theorem 5 in \cite{at} II.

We can state now the key result:

\begin{lemma}
\label{L_condbig} Assume that Hypotheses A and B are valid and let $g$ be
the function satisfying (\ref{As_f11}). Then,
\begin{eqnarray*}
&&\lim_{n\rightarrow 0}\bigg\vert b_{n}^{-1}\mathbf{E}[Y_{j}Y_{j,n}\left(
1-f_{0,n}(0)\right) e^{-\rho S_{n}};X_{j}\geq \delta an] \\
&&\quad \qquad -\mathbf{E}\left[ Y_{j}e^{-\rho S_{j-1}}\int_{-\infty
}^{\infty }Y_{n,j}(v)\left( 1-f_{0,j-1}(g_{j}(e^{v}W_{n,j}))\right) e^{-\rho
v}dv\right] \bigg\vert=0
\end{eqnarray*}%
where $(W_{n,j},f_{k}:k\geq j+1)$, $g_{j}$ and $(S_{j-1},f_{0,j-1})$ are
independent and%
\begin{eqnarray}
&&0<\lim_{n\rightarrow \infty }\mathbf{E}\left[ e^{-\rho
S_{j-1}}\int_{-\infty }^{\infty }\left(
1-f_{0,j-1}(g_{j}(e^{v}W_{n,j}))\right) e^{-\rho v}dv\right]  \notag \\
&&\qquad =\mathbf{E}\left[ e^{-\rho S_{j-1}}\int_{-\infty }^{\infty }\left(
1-f_{0,j-1}(g_{j}(e^{v}W_{j}))\right) e^{-\rho v}dv\right] <\infty .
\label{PositFin}
\end{eqnarray}
\end{lemma}

\begin{proof}
Introduce the event
\begin{equation*}
\mathcal{T}_{N,K,M}(j)=\left\{ \left\vert S_{j-1}\right\vert \leq
N,\,\left\vert S_{n}-S_{j-1}\right\vert \leq K,\,\left\vert
X_{j}-an\right\vert \leq M\sqrt{n}\right\} .
\end{equation*}

Recalling that $Y_{j}$ and $Y_{j,n}$ are bounded, to prove the lemma it is
sufficient to study only the quantity%
\begin{eqnarray*}
&&\mathbf{E}\left[ Y_{j}Y_{j,n}\left( 1-f_{0,n}(0)\right) e^{-\rho S_{n}};%
\mathcal{T}_{N,K,M}(j)\right] \\
&&\qquad =\mathbf{E}[Y_{j}Y_{n,j}[1-f_{0,j-1}(f_{j}(f_{n,j}(0)))]e^{-\rho
S_{j}}e^{-\rho S_{n,j}};\mathcal{T}_{N,K,M}(j)].
\end{eqnarray*}

Moreover, we may assume without loss of generality that $Y_{j}$ and $Y_{j,n}$
are nonnegative. The general case may be considered by writing $%
Y_{j}Y_{j,n}=\left( Y_{j}Y_{j,n}\right) ^{+}-\left( Y_{j}Y_{j,n}\right)
^{-}, $ where $x^{+}=\max \left( x,0\right) $ and $x^{-}=-\min
\left(x,0\right)$.

Clearly,%
\begin{equation*}
\left\{ X_{j}\geq an-M\sqrt{n},\left\vert S_{n}-S_{j-1}\right\vert \leq
K\right\} \subset \left\{ S_{n}-S_{j}\leq K-an+M\sqrt{n}\right\} .
\end{equation*}%
This, in view of the inequality
\begin{equation*}
e^{S_{n,j}}W_{n,j}=1-f_{n,j}(0)\leq e^{S_{n,j}}
\end{equation*}%
and the representation $e^{-x}=1-x+o(x),\,x\rightarrow 0,$ means that if the
event $\mathcal{T}_{N,K,M}(j)$ occurs then, for any $\varepsilon >0$ there
exists $n_{0}=n_{0}(\varepsilon )$ such that for all $n\geq n_{0}$%
\begin{equation*}
e^{-(1+\varepsilon )\left( 1-f_{n,j}(0\right) )}\leq f_{n,j}(0)\leq
e^{-\left( 1-f_{n,j}(0\right) )}.
\end{equation*}%
As a result we have%
\begin{eqnarray*}
&&\mathbf{E}\left[ Y_{j}Y_{n,j}\left( 1-f_{0,j-1}(f_{j}(e^{-\left(
1-f_{n,j}(0\right) )}))\right) e^{-\rho S_{j-1}}e^{-\rho S_{n,j-1}};\mathcal{%
T}_{N,K,M}(j)\right]  \\
&&\quad \leq b_{n}^{-1}\mathbf{E}\left[ Y_{j}Y_{j,n}\left(
1-f_{0,n}(0)\right) e^{-\rho S_{n}};\mathcal{T}_{N,K,M}(j)\right]  \\
&&\quad \leq \mathbf{E}\left[ Y_{j}Y_{n,j}\left(
1-f_{0,j-1}(f_{j}(e^{-\left( 1+\varepsilon \right) \left( 1-f_{n,j}(0\right)
)}))\right) e^{-\rho S_{j-1}}e^{-\rho S_{n,j-1}};\mathcal{T}_{N,K,M}(j)%
\right] .
\end{eqnarray*}%
Hence, denoting by $\mathcal{F}_{j-1}$ the $\sigma $-algebra generated by
the sequence
\begin{equation*}
\left( f_{1},...,f_{j-1};S_{1},...,S_{j-1}\right) ,
\end{equation*}%
we set%
\begin{equation*}
F_{n,j}(h,K,M;\varepsilon )=\mathbf{E}\left[ e^{-\rho (S_{n}-S_{j-1})}{Y}%
_{n,j}\bar{h}\left( f_{j}\left( \exp \left\{ -\left( 1+\varepsilon \right)
e^{S_{n}-S_{j}}W_{n,j}\right\} \right) \right) \,;B_{j,n}\right] ,
\end{equation*}%
\begin{equation*}
O_{n,j}(h,K,M;\varepsilon )=\int_{-K}^{K}e^{-\rho v}dv\mathbf{E}\left[
Y_{n,j}(v)\bar{h}\left( g_{j}\left( \left( 1+\varepsilon \right)
e^{v}W_{n,j}\right) \right) \,;\sigma G_{n,j}\in \left[ -M,M\right] \right]
\end{equation*}%
and introduce the random variables
\begin{equation*}
\hat{F}_{n,j}(f_{0,j-1},K,M;\varepsilon )=\mathbf{E}\left[
F_{n,j}(f_{0,j-1},K,M;\varepsilon )|\mathcal{F}_{j-1}\right]
\end{equation*}%
and%
\begin{equation*}
\hat{O}_{n,j}(f_{0,j-1},K,M;\varepsilon )=\mathbf{E}\left[
O_{n,j}(f_{0,j-1},K,M;\varepsilon )|\mathcal{F}_{j-1}\right] .
\end{equation*}%
We get from the previous inequalities
\begin{eqnarray}
&&\mathbf{E}\left[ Y_{j}e^{-\rho S_{j-1}}\hat{F}_{n,j}(f_{0,j-1},K,M;0);%
\left\vert S_{j-1}\right\vert \leq N\right]   \notag \\
&&\qquad \leq b_{n}^{-1}\mathbf{E}\left[ Y_{j}Y_{j,n}\left(
1-f_{0,n}(0)\right) e^{-\rho S_{n}};\mathcal{T}_{N,K,M}(j)\right]
\label{inek} \\
&&\qquad \quad \leq \mathbf{E}[Y_{j}e^{-\rho S_{j-1}}\hat{F}%
_{n,j}(f_{0,j-1},K,M;\varepsilon );\left\vert S_{j-1}\right\vert \leq N].
\notag
\end{eqnarray}%
Moreover the dominated convergence theorem and Lemma \ref{limO} give for any
fixed $\alpha \in \{0,\varepsilon \}$,
\begin{eqnarray*}
&&\limsup_{n\rightarrow \infty }\big\vert b_{n}^{-1}\mathbf{E}[Y_{j}e^{-\rho
S_{j-1}}\hat{F}_{n,j}(f_{0,j-1},K,M;\alpha );\left\vert S_{j-1}\right\vert
\leq N] \\
&&\qquad \qquad \qquad -\mathbf{E}[Y_{j}e^{-\rho S_{j-1}}\hat{O}%
_{n,j}(f_{0,j-1},K,M;\alpha );\left\vert S_{j-1}\right\vert \leq N]\big\vert%
=0.
\end{eqnarray*}%
Finally, $Y_{j}$ and $Y_{n,j}(v)$ are bounded (say by $1$ for convenience)
and we get
\begin{eqnarray*}
&&\limsup_{n\rightarrow \infty }\big\vert\mathbf{E}[Y_{j}e^{-\rho S_{j-1}}%
\hat{O}_{n,j}(f_{0,j-1},K,M;\varepsilon );\left\vert S_{j-1}\right\vert \leq
N] \\
&&\qquad \qquad \qquad -\mathbf{E}[Y_{j}e^{-\rho S_{j-1}}\hat{O}%
_{n,j}(f_{0,j-1},K,M;0);\left\vert S_{j-1}\right\vert \leq N]\big\vert \\
&&\quad \leq \lim \sup_{n\rightarrow \infty }\mathbf{E}\bigg[e^{-\rho
S_{j-1}}\int_{-K}^{K}e^{-\rho v}dv\mathbf{E}\big[f_{0,j-1}\big(g_{j}\left(
(1+\varepsilon )e^{-v}W_{n,j}\right) \big) \\
&&\qquad \qquad \qquad \qquad \qquad \qquad \qquad -f_{0,j-1}\left(
g_{j}\left( e^{-v}W_{n,j}\right) \right) \big];\left\vert S_{j-1}\right\vert
\leq N\bigg] \\
&&\quad =\mathbf{E}\bigg[e^{-\rho S_{j-1}}\int_{-K}^{K}e^{-\rho v}dv\mathbf{E%
}\big[f_{0,j-1}\left( g_{j}\left( (1+\varepsilon )e^{-v}W_{j}\right) \right)
- \\
&&\qquad \qquad \qquad \qquad \qquad \qquad \qquad f_{0,j-1}\left(
g_{j}\left( e^{-v}W_{j}\right) \right) \big];\left\vert S_{j-1}\right\vert
\leq N\bigg]
\end{eqnarray*}%
goes to $0$ as $\epsilon \rightarrow 0$ by monotonicity. We combine the last
limits with (\ref{inek}) to get
\begin{eqnarray}
&&\limsup_{n\rightarrow \infty }\big\vert b_{n}^{-1}\mathbf{E}\left[
Y_{j}Y_{n,j}\left( 1-f_{0,n}(0)\right) e^{-\rho S_{n}};\mathcal{T}_{N,K,M}(j)%
\right]   \notag \\
&&\qquad \qquad -\mathbf{E}[Y_{j}e^{-\rho S_{j-1}}\hat{O}%
_{n,j}(f_{0,j-1},K,M;0);\left\vert S_{j-1}\right\vert \leq N]\big\vert=0.
\label{Addit}
\end{eqnarray}%
By Corollary \ref{C_TimeSpace} and Lemmas \ref{Nlemma} (ii), \ref{L_Sj}, and %
\ref{L_remF2}, the fact that $Y_{j}$ and $Y_{n,j}$ are bounded ensure that
\begin{eqnarray}
&&\mathbf{E}\left[ Y_{j}Y_{j,n}\left( 1-f_{0,n}(0)\right) e^{-\rho
S_{n}};X_{j}\geq \delta an\right]   \notag \\
&=&\mathbf{E}[Y_{j}Y_{j,n}\left( 1-f_{0,n}(0)\right) e^{-\rho
S_{n}};\left\vert S_{j-1}\right\vert \leq N,X_{j}\geq \delta an]+\varepsilon
_{N,n}b_{n}  \notag \\
&=&\mathbf{E}[Y_{j}Y_{n,j}\left( 1-f_{0,n}(0)\right) e^{-\rho
S_{n}};\left\vert S_{j-1}\right\vert \leq N,\left\vert
S_{n}-S_{j-1}\right\vert \leq K,X_{j}\geq \delta an]+\varepsilon
_{N,K,n}b_{n}  \notag \\
&=&\mathbf{E}[Y_{j}Y_{n,j}\left( 1-f_{0,n}(0)\right) e^{-\rho S_{n}};%
\mathcal{T}_{N,K,M}(j)]+\varepsilon _{N,K,M,n}(j)b_{n},  \label{MUMU1}
\end{eqnarray}%
where
\begin{equation}
\lim_{N\rightarrow \infty }\limsup_{K\rightarrow \infty
}\limsup_{M\rightarrow \infty }\limsup_{n\rightarrow \infty }\left\vert
\varepsilon _{N,K,M,n}(j)\right\vert =0.  \label{MUMU2}
\end{equation}%
Taking now $Y_{j}=Y_{n,j}\equiv 1$, adding that
$\mathbf{E}\left[ \left( 1-f_{0,n}(0)\right) e^{-\rho
S_{n}}\right] =O\left( b_{n}\right) $ by Corollary \ref{C_bound}
and recalling (\ref{Addit}), we deduce, again by monotonicity
that \begin{eqnarray*} &&\lim_{N\rightarrow
\infty }\lim_{K\rightarrow \infty }\lim_{M\rightarrow
\infty }\limsup_{n\rightarrow \infty }\mathbf{E}[e^{-\rho S_{j-1}}\hat{O}%
_{n,j}(f_{0,j-1},K,M;0);\left\vert S_{j-1}\right\vert \leq N] \\
&&\qquad =\mathbf{E}\left[ e^{-\rho S_{j-1}}\int_{-\infty }^{\infty }\left(
1-f_{0,j-1}(g_{j}(e^{v}W_{j}))\right) e^{-\rho v}dv\right]  \\
&&\qquad \leq \limsup_{n\rightarrow \infty }b_{n}^{-1}\mathbf{E}\left[
\left( 1-f_{0,n}(0)\right) e^{-\rho S_{n}}\right] \leq C<\infty ,
\end{eqnarray*}
proving, in particular, the estimate from above in
(\ref{PositFin}). This, in turn, implies for arbitrary uniformly
bounded $Y_{j}$ and $Y_{n,j}$,
\begin{eqnarray*}
&&\limsup_{n\rightarrow \infty }\mathbf{E}[Y_{j}e^{-\rho S_{j-1}}\hat{O}%
_{n,j}(f_{0,j-1},\infty ,\infty ;0)] \\
&&\qquad \qquad \leq C\mathbf{E}\left[ e^{-\rho S_{j-1}}\int_{-\infty
}^{\infty }\left( 1-f_{0,j-1}(g_{j}(e^{v}W_{j}))\right) e^{-\rho v}dv\right]
<\infty
\end{eqnarray*}%
and
\begin{eqnarray*}
&&\limsup_{n\rightarrow \infty }\big\vert b_{n}^{-1}\mathbf{E}\left[
Y_{j}Y_{j,n}\left( 1-f_{0,n}(0)\right) e^{-\rho S_{n}};X_{j}\geq \delta an%
\right] - \\
&&\qquad \qquad \qquad \qquad \qquad \qquad -\mathbf{E}[Y_{j}e^{-\rho
S_{j-1}}\hat{O}_{n,j}(f_{0,j-1},\infty ,\infty ;0)]\big\vert=0.
\end{eqnarray*}%
It yields the first part of the Lemma. We have already checked
the finiteness of the limit in (\ref{PositFin}).
Positivity follows from conditions (\ref{Mom}), since under these
conditions $W>0$ with probability 1 according to
Theorem~5~\cite{at},~II.  This gives the whole result.
\end{proof}

\section{Proof of the Theorems and the Corollary}

\label{ProofMT}

Now we have an important corollary, which, in fact, proves Theorem \ref%
{T_Extin} with the explicit form of the constant $C_{0}$ mentioned in the
statement of the theorem. \newline

\begin{proof}[Proof of Theorem 1]
We assume that Hypotheses A and B are valid. It follows from (\ref%
{Remainder1}) that for each fixed $j$%
\begin{equation*}
\mathbf{E}\left[ (1-f_{0,n}(0))\exp (-\rho S_{n});\mathcal{D}_{N}(j)\right] =%
\mathbf{E}[\left( 1-f_{0,n}(0)\right) e^{-\rho S_{n}};X_{j}\geq \delta
an]+\varepsilon _{N,n}b_{n}.
\end{equation*}%
Using this fact and Lemmas \ref{L_condbig} and \ref{L_Negl2} we get%
\begin{eqnarray*}
&&\lim_{n\rightarrow \infty }m^{-n}b_{n}^{-1}\mathbb{P}\left( Z_{n}>0\right)
=\lim_{n\rightarrow \infty }m^{-n}b_{n}^{-1}\mathbb{E}\left[ (1-f_{0,n}(0))%
\right] \\
&&\qquad =\lim_{n\rightarrow \infty }b_{n}^{-1}\mathbf{E}\left[
(1-f_{0,n}(0))\exp (-\rho S_{n})\right] =C_{0}.
\end{eqnarray*}%
To complete the proof it remains to observe that in view of (\ref{Tail1})
\begin{eqnarray*}
b_{n} &=&\beta \frac{\mathbf{P}\left( X>an\right) }{an}\sim \frac{1}{m}\frac{%
l_{0}(an)}{\left( an\right) ^{\beta +1}}=\frac{1}{m}\frac{l_{0}(an)}{\left(
an\right) ^{\beta +1}}e^{-\rho an}e^{\rho an} \\
&\sim &\frac{\rho }{m}e^{\rho an}\int_{an}^{\infty }p_{X}(x)dx=\frac{\rho }{m%
}e^{\rho an}\mathbb{P}\left( X>an\right) .
\end{eqnarray*}%
Thus,
\begin{equation*}
\mathbb{P}\left( Z_{n}>0\right) \sim C_{0}m^{n}b_{n}\sim C_{0}\rho m^{n-1}%
\mathbb{P}\left( X>an\right) e^{an\rho }
\end{equation*}%
where, recalling that $g_{j},W_{j}$ and $f_{0,j-1}$ are independent
\begin{eqnarray}
C_{0} &=&\sum_{j=1}^{\infty }\mathbf{E}\left[ e^{-\rho S_{j-1}}\int_{-\infty
}^{\infty }\left( 1-f_{0,j-1}(g_{j}(e^{v}W_{j}))\right) e^{-\rho v}dv\right]
\label{C_finsurviv} \\
&=&\sum_{j=1}^{\infty }\left( \mathbb{E}\left[ e^{\rho X}\right] \right)
^{-j+1}\int_{-\infty }^{\infty }\mathbb{E}\left[
1-f_{0,j-1}(g_{j}(e^{v}W_{j}))\right] e^{-\rho v}dv.  \notag
\end{eqnarray}%
The proof of the first Theorem is achieved.
\end{proof}

$\newline
$

\begin{proof}[Proof of Theorem 2]
Let
\begin{equation*}
W_{n,j}(s)=\frac{1-f_{n,j}(s)}{e^{S_{n}-S_{j}}},\,s\in \lbrack 0,1).
\end{equation*}%
By monotonicity
\begin{equation*}
\lim_{n\rightarrow \infty }W_{n,j}(s)=W_{j}(s)
\end{equation*}%
and $W_{j}(s)\overset{d}{=}W(s),\,j=1,2,...$ where $\mathbf{P}(W(s)\in
(0,1])=1$ thanks to \cite{at} II Theorem 5.

Similarly to Lemma \ref{L_condbig} one can show that, as $n\rightarrow
\infty $%
\begin{eqnarray*}
&&\lim_{n\rightarrow \infty }b_{n}^{-1}\mathbf{E}[\left( 1-f_{0,n}(s)\right)
e^{-\rho S_{n}}] \\
&&\quad=\lim_{n\rightarrow \infty }b_{n}^{-1}\sum_{j=1}^{\infty }\mathbf{E}%
[\left( 1-f_{0,n}(s)\right) e^{-\rho S_{n}};X_{j}\geq \delta an] \\
&&\quad=\mathbf{E}\left[ e^{-\rho S_{j-1}}\int_{-\infty }^{\infty }\left(
1-f_{0,j-1}(g(e^{v}W(s)))\right) e^{-\rho v}dv\right] =\Omega _{0}(s).
\end{eqnarray*}

Hence we get%
\begin{eqnarray*}
\lim_{n\rightarrow \infty }\mathbb{E}\left[ s^{Z_{n}}|Z_{n}>0\right]
&=&1-\lim_{n\rightarrow \infty }\frac{\mathbf{E}[\left( 1-f_{0,n}(s)\right)
e^{-\rho S_{n}}]}{\mathbf{E}[\left( 1-f_{0,n}(0)\right) e^{-\rho S_{n}}]} \\
&=&1-C_{0}^{-1}\Omega _{0}(s)=:\Omega (s).
\end{eqnarray*}
Theorem \ref{T_condEnd} is proved.
\end{proof}

$\newline
$

\begin{proof}[Proof of Theorem 3]
Coming back to the original probability $\mathbb{P}$, Lemma \ref{L_condbig}
yields
\begin{eqnarray*}
&&\lim_{n\rightarrow \infty }\bigg\vert b_{n}^{-1}m^{-n}\mathbb{E}%
[Y_{j}Y_{j,n}\mathbb{P}_{\mathfrak{e}}(Z_{n}>0);X_{j}\geq \delta an] \\
&&\quad \qquad -m^{-{j-1}}\mathbb{E}\left[ Y_{j}\int_{-\infty }^{\infty
}Y_{n,j}(v)\left( 1-f_{0,j-1}(g_{j}(e^{v}W_{n,j}))\right) e^{-\rho v}dv%
\right] \bigg\vert=0
\end{eqnarray*}

Recalling that $\mathbb{P}(Z_{n}>0)\sim C_{0}m^{n}b_{n}$ as $n\rightarrow
\infty$ ensures that
\begin{eqnarray*}
&&\lim_{n\rightarrow \infty }\bigg\vert\mathbb{E}[Y_{j}Y_{j,n};X_{j}\geq
\delta an|Z_{n}>0] \\
&&\quad \qquad -C_{0}^{-1}m^{-{j-1}}\mathbb{E}\left[ Y_{j}\int_{-\infty
}^{\infty }Y_{n,j}(v)\left( 1-f_{0,j-1}(g_{j}(e^{v}W_{n,j}))\right) e^{-\rho
v}dv\right] \bigg\vert=0
\end{eqnarray*}%
We obtain the first part of the Theorem by letting $Y_{j}=1$ and $Y_{j,n}=1$
and using (\ref{monot}), whereas the second part comes by dividing the last %
displayed formula by $\mathbb{P}(X_{j}\geq \delta
an|Z_{n}>0)$.
\end{proof}

$\newline
$

\begin{proof}[Proof of the Corollary]
We first check that conditionally on $Z_{n}>0$, there is only one big jump.
Recalling from Section 5.2 the notation $\mathcal{C}_{N}=\left\{ -N<S_{\tau
_{n}}\leq S_{n}\leq N+S_{\tau _{n}}<N\right\} $ and the inequality $\mathbf{P%
}_{\mathfrak{e}}\left( Z_{n}>0\right) \exp (-\rho S_{n})\leq~1$ justified by
(\ref{Le1}) we have
\begin{eqnarray*}
&&\mathbb{P}(Z_{n}>0,\cup _{i\neq j}^{n}\{X_{i}\geq an/2,X_{j}\geq an/2\}) \\
&&\quad = m^{n}\mathbf{E}\left[ \mathbf{P}_{\mathfrak{e}}\left(
Z_{n}>0\right) \exp (-\rho S_{n});\cup _{i\neq j}^{n}\{X_{i}\geq
an/2,X_{j}\geq an/2\}\right] \\
&& \quad \leq m^{n}\bigg(\mathbf{E}\left[ \mathbf{P}_{\mathfrak{e}}\left(
Z_{n}>0\right) \exp (-\rho S_{n});\mathcal{\bar{C}}_{N}\right] \\
&& \qquad \qquad \qquad +\mathbf{E}\left[ L_{n}\geq -N,S_{n}\geq N,\cup
_{i\neq j}^{n}\{X_{i}\geq an/2,X_{j}\geq an/2\}\right] \bigg).
\end{eqnarray*}%
Then Lemma \ref{L_SmallEnd} and the limiting relation (\ref{twojumps})
ensure that
\begin{equation*}
\limsup_{n\rightarrow \infty }(b_{n}m^{n})^{-1}\mathbb{P}\left( Z_{n}>0,\cup
_{i\neq j}^{n}\{X_{i}\geq an/2,X_{j}\geq an/2\}\right) =0.
\end{equation*}%
Thus, $\lim_{n\rightarrow \infty }\mathbb{P}(\cup _{i\neq j}^{n}\{X_{i}\geq
an/2,X_{j}\geq an/2\}|Z_{n}>0)=0$. The first part of the Corollary is then a
direct consequence of Theorem 3 (i).

Since $X_{j}=(S_{n}-S_{j-1})-(X_{j+1}+\ldots +X_{n})$, the second part is
obtained from Theorem 3 (ii) with $F(\cdot )=1$, $F_{n-j}(v,x_{j+1},\ldots
,x_{n})=H((v-x_{j+1}\ldots -x_{n}-an)/\sqrt{n})$, where $H$ is measurable
and bounded.
\end{proof}

$\newline
$

\textbf{Acknowledgement.} This work was partially funded by the project
MANEGE `Mod\`{e}les Al\'{e}atoires en \'{E}cologie, G\'{e}n\'{e}tique et
\'{E}volution' 09-BLAN-0215 of ANR (French national research agency), Chair
Modelisation Mathematique et Biodiversite VEOLIA-Ecole
Polytechnique-MNHN-F.X. and the professorial chair Jean Marjoulet. The
second author was also supported by the Program of the Russian Academy of
Sciences \textquotedblleft Dynamical systems and control
theory\textquotedblright .


\begin{thebibliography}{99}
\bibitem{af98} \textsc{Afanasyev V.\thinspace I.} (1998). Limit theorems for
a moderately subcritical branching process in a random environment. \textsl{%
Discrete Math. Appl., }\textbf{8}, pp. 55--62.

\bibitem{abkv1} \textsc{Afanasyev V.I., Boeinghoff C., Kersting G., and
Vatutin V.A. }(2012) Limit theorems for weakly subcritical branching
processes in random environment. \textsl{J.Theor.Probab.}\textit{,} \textbf{%
25, }N 3, pp. 703--732.

\bibitem{abkv2} \textsc{Afanasyev V.I., Boeinghoff C., Kersting G., and
Vatutin V.A. }(2013) Conditional limit theorems for intermediately
subcritical branching processes in random environment. \textsl{Ann. Inst. H.
Poincar\'e Probab. Statist.}, In print, arXiv:1108.2127 [math.PR]

\bibitem{agkv} \textsc{Afanasyev V.\thinspace I., Geiger J., Kersting G.,
and Vatutin V.\thinspace A.} (2005). Criticality for branching processes in
random environment. \textsl{Ann. Probab.} \textbf{33}, pp.645--673.

\bibitem{agkv2} \textsc{Afanasyev V.\thinspace I., Geiger J., Kersting G.,
and Vatutin V.\thinspace A.} (2005). Functional limit theorems for strongly
subcritical branching processes in random environment. \textsl{Stoch. Proc.
Appl., }\textbf{115}, pp.1658--1676.

\bibitem{at} \textsc{Athreya K.B., and Karlin S.} (1971). On branching
processes with random environments: I, II, \textsl{Ann. Math. Stat., }
\textbf{42}, pp.1499--1520, pp.1843--1858.


\bibitem{VVI} \textsc{Bansaye V., and Vatutin V.} (2013). Random walk with
heavy tail and negative drift conditionned by its minimum and final values.
Avialable via \emph{Arxiv}, http://arxiv.org/abs/1312.3306.

\bibitem{BGT} \textsc{Bingham N.H., Goldie C.M., and Teugels J.L.} (1987).
\textsl{Regular variation}. Cambridge University Press, Cambridge.

\bibitem{bgk} \textsc{Birkner M., Geiger J., and Kersting G.} (2005).
Branching processes in random environment - a view on critical and
subcritical cases. Proceedings of the DFG-Schwerpunktprogramm \textsl{%
Interacting Stochastic Systems of High Complexity}, Springer, Berlin,
265--291.

\bibitem{BB2005} \textsc{Borovkov A.A., and Borovkov K.A. }(2008). \textsl{%
Asymptotic analysis of random walks. Heavy-tailed distributions.}
Encyclopedia of Mathematics and its Applications, 118. Cambridge University
Press, Cambridge.

\bibitem{cnw} \textsc{Chover J., Ney P., and Wainger S.} (1973). Functions
of Probability measures. \textsl{J. Analyse Math., } \textbf{26}, pp.
255--302.

\bibitem{Durr} \textsc{Durrett, R.} (1980). Conditioned limit theorems for
random walks with negative drift. \textsl{Z. Wahrsch. Verw. Gebiete} 52, no.
\textbf{3}, 277-287.

\bibitem{fe} \textsc{Feller W.} (1971). \textsl{An Introduction to
Probability Theory and Its Applications, Volume II}. John Wiley and Sons,
New York.

\bibitem{gkv} \textsc{Geiger J., Kersting G., and Vatutin V.A.} (2003).
Limit theorems for subcritical branching processes in random environment.
\textsl{Ann. I.H. Poincar\'{e} (B). }\textbf{39}, pp. 593--620.

\bibitem{Hir98} \textsc{Hirano K.} (1998). Determination of the Limiting
Coefficient for Exponential Functionals of Random Walks with Positive Drift.
\textsl{J. Math. Sci. Univ. Tokyo,} \textbf{5} , pp. 299--332.

\bibitem{SmWil} \textsc{Smith W.L., and Wilkinson W.E.} (1969). On branching
processes in random environments. \textsl{Ann. Math. Stat., }\textbf{40},
pp. 814--827.

\bibitem{vz} \textsc{Vatutin V., and Zheng X.} (2012). Subcritical branching
processes in random environment without Cramer condition. \textsl{Stochastic
Process. Appl.}, \textbf{122}, pp. 2594--2609.
\end{thebibliography}
\end{document}